\documentclass[12pt]{article}
\usepackage{geometry}
\usepackage[utf8]{inputenc}
\usepackage[T1]{fontenc}
\usepackage{amsmath}
\usepackage{amsthm}
\usepackage{amssymb}
\usepackage{mathrsfs}
\usepackage{array}
\usepackage{stmaryrd}
\usepackage{graphicx}
\usepackage[backend=biber,maxalphanames=4,maxnames=4,style=alphabetic,sorting=nyt,giveninits]{biblatex}
\usepackage[linktoc = page]{hyperref}
\usepackage{url}
\usepackage{bbm}
\usepackage{bm}
\usepackage{comment}
\usepackage[table,xcdraw,dvipsnames]{xcolor}
\usepackage{etoolbox}
\usepackage{tikz-cd}

\newcounter{continue-counter}

\newcommand{\N}{\mathbb{N}}
\newcommand{\Z}{\mathbb{Z}}
\newcommand{\Q}{\mathbb{Q}}

\newcommand{\SymGroup}[1][n]{\mathfrak{S}_{#1}}
\newcommand{\Specht}[1]{S_{#1}}
\newcommand{\TransposePartition}[1]{{#1}^{\ast}}

\newcommand{\into}{\hookrightarrow}
\newcommand{\onto}{\twoheadrightarrow}

\newcommand{\Catgr}{\mathbf{gr}}
\newcommand{\TensorPower}[2]{#1^{\otimes {#2}}}

\newcommand{\CatFuncGrinit}[4][]{%
    \ifstrempty{#1}{%
        \mathcal{F}^{#2}_{#3}(\Catgr^{#4})
    }{%
        \mathcal{F}^{#2}_{#3}(\Catgr^{#4}, {#1})
    }
}
\newcommand{\CatFuncGraux}[3][]{\CatFuncGrinit[{#1}]{#2}{#3}{}}
\newcommand{\CatFuncGr}[1][]{\CatFuncGraux[{#1}]{}{}}
\newcommand{\CatOutFuncGr}[1][]{\CatFuncGraux[{#1}]{\mathrm{out}}{}}
\newcommand{\CatPolyFunc}[2][]{\CatFuncGraux[{#1}]{}{#2}}
\newcommand{\CatPolyOutFunc}[2][]{\CatFuncGraux[{#1}]{\mathrm{out}}{#2}}

\newcommand{\Foutpol}{\CatPolyOutFunc{\mathrm{pol}}}

\newcommand{\LeftAdjq}[1]{\mathbf{q}_{#1}}

\newcommand{\Vect}[1]{\textrm{Vect}_{#1}}

\newcommand{\ab}{\mathfrak{a}}

\newcommand{\FreeGroup}[1][n]{\Z^{\star{#1}}}

\DeclareMathOperator{\Aut}{Aut}

\DeclareMathOperator{\Ext}{Ext}

\DeclareMathOperator{\Hom}{Hom}
\DeclareMathOperator{\Id}{Id}
\DeclareMathOperator{\Lie}{Lie}
\DeclareMathOperator{\Out}{Out}

\theoremstyle{plain}
\newtheorem{thm}{Theorem}[section]
\newtheorem{lem}[thm]{Lemma}
\newtheorem{prop}[thm]{Proposition}
\newtheorem{conj}[thm]{Conjecture}
\newtheorem{cor}[thm]{Corollary}

\theoremstyle{definition}
\newtheorem{defi}[thm]{Definition}
\newtheorem{expl}[thm]{Example}

\theoremstyle{remark}
\newtheorem{rmk}[thm]{Remark}

\newtheorem{terminology}[thm]{Terminology}

\title{(Non-)vanishing results for extensions between simple outer functors on free groups}
\author{Louis Hainaut}

\begin{document}

\maketitle

\begin{abstract}
    In this article we study cohomological properties of the category of \emph{polynomial outer functors} on free groups, which are the functors from the category of finitely generated free groups to the category of rational vector spaces which send all inner automorphisms to the identity morphism, and which satisfy a certain polynomiality property. More precisely, we prove vanishing and non-vanishing results for the $\Ext$ groups between simple polynomial outer functors. This work is inspired by an earlier result of Vespa for the category of all polynomial functors from finitely generated free groups to rational vector spaces; it follows in particular from her results that, in this larger category, the $\Ext$ groups between simple functors are always concentrated in a specific single degree. Our main results show that, when we pass to the full subcategory of polynomial outer functors, $\Ext$ groups between simple functors are sometimes non-trivial outside of this specific degree.
\end{abstract}

\section{Introduction}
Let $\Catgr$ be the category of finitely-generated free groups and group morphisms, and let $\Vect{\Q}$ be the category of finite-dimensional rational vector spaces. Furthermore let $\CatFuncGr$ be the category of functors $F\colon \Catgr\to \Vect{\Q}$. As we will discuss below, the objects of this functor category provide natural families of representations of $\Aut(\FreeGroup)$, the group of all automorphisms of the free group $\FreeGroup\in\Catgr$.

An important family of functors in $\CatFuncGr$ is given by the functors
\begin{equation*}
    \TensorPower{\ab}{d}\colon G \mapsto (G_{ab}\otimes \Q)^{\otimes d} = H_1(G, \Q)^{\otimes d}.
\end{equation*}
More explicitly, these functors are defined on objects by $\TensorPower{\ab}{d}(\FreeGroup) = (\Q^n)^{\otimes d}$. The functors $\TensorPower{\ab}{d}$ are equipped with an action of the symmetric group $\SymGroup[d]$ on the right, induced by the usual \emph{place permutation} action: $(v_1\otimes\ldots\otimes v_d)\bullet\sigma = v_{\sigma(1)}\otimes\ldots\otimes v_{\sigma(d)}$.
Vespa computed in \cite{Ves18} the groups $\Ext^*_{\CatFuncGr}(\TensorPower{\ab}{n}, \TensorPower{\ab}{m})$, obtaining in particular that these vanish for $*\neq m-n$ and describing the structure of $\Ext^{m-n}_{\CatFuncGr}(\TensorPower{\ab}{n}, \TensorPower{\ab}{m})$ as a $(\Q\SymGroup[n],\Q\SymGroup[m])$-bimodule.

A functor $F\in \CatFuncGr$ is called \emph{polynomial} if certain \emph{cross-effect functors} associated to $F$ vanish (Definition \ref{def:polynomial-functors}). The \emph{degree} of $F$ is then defined to be the maximal index of a non-vanishing cross-effect functor. By functoriality, for any functor $F\in\CatFuncGr$ evaluation at $\FreeGroup$ gives a $\Aut(\FreeGroup)$-representation $F(\FreeGroup)$. Powell and Vespa introduced in \cite{PV18} the full subcategory $\CatOutFuncGr$ of \emph{outer functors}, whose objects are those functors $F$ for which the inner automorphisms $\operatorname{Inn}(\FreeGroup)$ act trivially on $F(\FreeGroup)$, or in other words $F(\FreeGroup)$ is a representation of $\Out(\FreeGroup)$. 
Since its introduction, this category has been observed to play an important role in multiple parts of mathematics, including Hochschild--Pirashvili (co)homology \cite{TW19, PV18}, string topology as well as embedding and mapping spaces \cite{TW19}, configuration spaces on a wedge sum of circles \cite{GH22}, cohomology of (tropical) moduli spaces of curves \cite{BCGY21, BCGY23} and Jacobi diagrams in handlebodies \cite{Katada23, Katada21, Ves22}. It has been observed that some of the outer functors studied there provide examples of highly non-trivial representations of $\Out(\FreeGroup)$.

We will denote by $\CatPolyOutFunc{\mathrm{pol}}$ the subcategory of \emph{polynomial outer functors}. The functors $\TensorPower{\ab}{d}$ are examples of polynomial outer functors, of degree exactly $d$. Our goal in this paper is to study the groups $\Ext^*_{\CatPolyOutFunc{\mathrm{pol}}}(\TensorPower{\ab}{n}, \TensorPower{\ab}{m})$. The main result we obtain is the following:

\begin{thm}\label{thm:intro-vanishing-and-non-koszul}
    Let $n,m$ be two non-negative integers.
    \begin{enumerate}
        \item The groups $\Ext^k_{\CatPolyOutFunc{\mathrm{pol}}}(\TensorPower{\ab}{n}, \TensorPower{\ab}{m})$ vanish for all $k>m-n$.
        \item The group $\Ext^k_{\CatPolyOutFunc{\mathrm{pol}}}(\TensorPower{\ab}{n}, \TensorPower{\ab}{m})$ does not always vanish when $k<m-n$. For example $\Ext^4_{\CatPolyOutFunc{\mathrm{pol}}}(\TensorPower{\ab}{4}, \TensorPower{\ab}{9}) \neq 0$.
    \end{enumerate}
\end{thm}
The same arguments also show that among the groups $\Ext^k_{\CatPolyOutFunc{\mathrm{pol}}}(\TensorPower{\ab}{3}, \TensorPower{\ab}{10})$ with $k < 7$ at least one of them is non-trivial.

\begin{rmk}
    Vespa's theorem \cite{Ves18} that $\Ext^*_{\CatFuncGr}(\TensorPower{\ab}{n},\TensorPower{\ab}{m})$ is concentrated in degree $*=m-n$ was interpreted by Powell in \cite{Pow21} as a certain Koszul-type property for the category $\CatFuncGr$. The second item in the theorem above shows that $\Foutpol$ does not have this Koszul-type property.

    It happens however that for certain values of $n$ and $m$ the groups $\Ext^*_{\Foutpol}(\TensorPower{\ab}{n},\TensorPower{\ab}{m})$ are concentrated in degree $* = m-n$. When this is the case, we will by analogy say that, for that specific pair $(n, m)$, $\Ext^*_{\Foutpol}(\TensorPower{\ab}{n},\TensorPower{\ab}{m})$ has the \emph{Koszul property}.
\end{rmk}

We also prove (some of the notation used here is explained later):

\begin{thm}
    For any $n\geq 4$, we have $\Ext^*_{\CatPolyOutFunc{\mathrm{pol}}}(\TensorPower{\ab}{2}, \TensorPower{\ab}{n}) = 0$. The only non-zero values are the $(\Q\SymGroup[n],\Q\SymGroup[m])$-bimodules $\Ext^0_{\CatPolyOutFunc{\mathrm{pol}}}(\TensorPower{\ab}{2}, \TensorPower{\ab}{2}) \cong \Q\SymGroup[2]$ and $\Ext^1_{\CatPolyOutFunc{\mathrm{pol}}}(\TensorPower{\ab}{2}, \TensorPower{\ab}{3}) \cong \Specht{(2)}\boxtimes\Specht{(1^3)}$.
\end{thm}

Note that $\Ext^*_{\CatPolyOutFunc{\mathrm{pol}}}(\Q, \TensorPower{\ab}{n}) = 0$ for all $n\neq 0$ and $\Ext^*_{\CatPolyOutFunc{\mathrm{pol}}}(\ab, \TensorPower{\ab}{n}) = 0$ for all $n\neq 1$ since the constant functor $\Q$ and the functor $\ab$ are projective in $\CatPolyOutFunc{\mathrm{pol}}$. The striking feature of the above theorem is that we have a vanishing result even though $\TensorPower{\ab}{2}$ is not projective (and indeed, this non-projectivity is witnessed by $\Ext^1_{\CatPolyOutFunc{\mathrm{pol}}}(\TensorPower{\ab}{2}, \TensorPower{\ab}{3})\neq 0$).

An important family of polynomial outer functors are the functors denoted $\omega\beta\Q\SymGroup[n]$ for $n\geq 0$, with $\Q\SymGroup[n]$ denoting the regular representation of $\SymGroup[n]$. It was proved in \cite{PV18} that the functor $\omega\beta\Q\SymGroup[n]$ is the injective envelope in the category $\CatPolyOutFunc{\mathrm{pol}}$ of the functor $\TensorPower{\ab}{n}$ (cf Corollary \ref{cor:simple-and-injective-outer}), and this family of functors appears naturally in some of the contexts mentioned above. In particular understanding these functors better, starting with their composition factors, is equivalent to an open problem posed in \cite{TW19}. Computing $\Ext^*_{\CatPolyOutFunc{\mathrm{pol}}}(\TensorPower{\ab}{n}, \TensorPower{\ab}{m})$ as a $(\Q\SymGroup[n], \Q\SymGroup[m])$-bimodule for all $n, m$ would allow to compute all these composition factors, as we will now explain.

Any functor $F\in \CatPolyFunc{\textrm{pol}}$ admits a canonical \emph{polynomial filtration}, having the property that the $p$-th graded piece of this filtration, which we denote\footnote{This definition of $F^{[p]}$ causes us to commit some abuses of notation in this introduction; we will redefine them in a slightly different (but essentially equivalent) way in the next section.} $F^{[p]}$, is homogeneous of degree $p$ and semi-simple; the simple functors appearing as summands of the various $F^{[p]}$ are the \emph{composition factors} of $F$. We will use in the following that $(\omega\beta\Q\SymGroup[d])^{[t]} = 0$ if $t > d$, and that $(\omega\beta\Q\SymGroup[d])^{[d]} = \TensorPower{\ab}{d}$.

\begin{prop}\label{prop:intro-koszul-determined}
    The graded pieces $(\omega\beta\Q\SymGroup[d])^{[p]}$ for all integers $d$ and $p$ can be determined from the $(\Q\SymGroup[n], \Q\SymGroup[m])$-bimodules $\Ext^*_{\CatPolyOutFunc{\mathrm{pol}}}(\TensorPower{\ab}{n}, \TensorPower{\ab}{m})$ for all $p\leq n \leq m \leq d$.
\end{prop}

We will prove (a stronger version of) this result in Theorem \ref{thm:determine-composition-factors}. However we must note that, at the time of writing, the result as stated here is of limited practical value, since $\Ext^k_{\CatPolyOutFunc{\mathrm{pol}}}(\TensorPower{\ab}{n}, \TensorPower{\ab}{m})$ have mostly unknown values: before this paper only results about $\Ext^0$ and $\Ext^1$ have appeared in the literature (see Proposition \ref{prop:values-Ext0-1}).

On the other hand, the graded pieces $(\omega\beta\Q\SymGroup[d])^{[p]}$ are currently better understood, and in joint work with Gadish \cite{GH22} the author completely computed them in the range $d\leq 10$.

In fact, for this paper the transferral of information goes in the other direction: indeed for Proposition \ref{prop:intro-koszul-determined} one only needs the equivariant Euler characteristics $\chi(\Ext^*_{\Foutpol}(\TensorPower{\ab}{n},\TensorPower{\ab}{m}))$, and it turns out that these are uniquely determined\footnote{The algorithms used to compute these Euler characteristics, as well as the result of some computations, can be accessed here: \url{https://github.com/louishainaut/Ext-Outer-Functors}.} by the graded pieces $(\omega\beta\Q\SymGroup[d])^{[p]}$. In some favorable cases we can even use the techniques presented here to obtain information about the individual groups $\Ext^*_{\Foutpol}(\TensorPower{\ab}{n},\TensorPower{\ab}{m})$. For example, if $\Ext^*_{\Foutpol}(\TensorPower{\ab}{n},\TensorPower{\ab}{m})$ has the Koszul property then $\chi(\Ext^*_{\Foutpol}(\TensorPower{\ab}{n},\TensorPower{\ab}{m})) = (-1)^{m-n}[\Ext^{m-n}_{\Foutpol}(\TensorPower{\ab}{n},\TensorPower{\ab}{m})]$.

Our main tool is the following spectral sequence:

\begin{prop}
    For each pair of integers $n, m$ there exists a spectral sequence of $(\Q\SymGroup[n],\Q\SymGroup[m])$-bimodules
    \begin{equation*}
        E_1^{p,q} = \Ext^{p+q}_{\Foutpol}(\TensorPower{\ab}{n},\TensorPower{\ab}{p})\otimes_{\SymGroup[p]}(\omega\beta\Q\SymGroup[m])^{[p]} \Rightarrow \Ext^{p+q}_{\Foutpol}(\TensorPower{\ab}{n},\omega\beta\Q\SymGroup[m]).
    \end{equation*}
    Moreover, if $n\neq m$ this spectral sequence converges to $0$.
\end{prop}

Combining this convergence to $0$ with the computations of $(\omega\beta\Q\SymGroup[m])^{[p]}$ performed by Gadish and the author in \cite{GH22}, we prove Theorem \ref{thm:intro-vanishing-and-non-koszul}. 

\begin{rmk}
    Combining Proposition \ref{prop:values-Ext0-1} and Theorem \ref{thm:intro-vanishing-and-non-koszul} we see that $\Ext^k_{\CatPolyOutFunc{\mathrm{pol}}}(\TensorPower{\ab}{n},\TensorPower{\ab}{n+2})$ vanishes for all $k\neq 2$, and we perform explicit computations for $k=2$ and $n\leq 8$ in Corollary \ref{cor:computing-Ext2}. More conceptually, Geoffrey Powell pointed out to the author that there is a structure of graded PROP on $\Ext^*_{\CatPolyOutFunc{\mathrm{pol}}}(\TensorPower{\ab}{n}, \TensorPower{\ab}{m})$ (similar to the structure discussed in \cite{Ves18}) and so the identity morphism $\Id\in \Hom(\ab, \ab)$ induces a stabilization map
\begin{equation*}
    \Ext^k_{\CatPolyOutFunc{\mathrm{pol}}}(\TensorPower{\ab}{n}, \TensorPower{\ab}{m}) \to \Ext^k_{\CatPolyOutFunc{\mathrm{pol}}}(\TensorPower{\ab}{n+1}, \TensorPower{\ab}{m+1}).
\end{equation*}
Our computations suggest that under this stabilization map $\Ext^2_{\CatPolyOutFunc{\mathrm{pol}}}(\TensorPower{\ab}{n}, \TensorPower{\ab}{n+2})$ could be generated by $n=3$ and $n=4$.
\end{rmk}

This introduction would not be complete without a discussion about homological stability: Vespa's Theorem \cite{Ves18} can be combined with a result of Djament \cite{Dja19} to obtain a computation of the stable cohomology of $\Aut(\FreeGroup)$ with certain twisted coefficients (specifically: for local systems obtained from the (covariant) functors $\TensorPower{\ab}{n}$). One could hope that it is possible to find a similar relation between the $\Ext$-groups computed in this paper and the stable cohomology of $\Out(\FreeGroup)$ with twisted coefficients, but no analogue of Djament's result is currently known. In any case, for these twisted coefficients, the stable cohomology of both $\Aut(\FreeGroup)$ and $\Out(\FreeGroup)$ has already been computed by Randal-Williams \cite{Randal18}, using geometric techniques inspired by \cite{Galatius11}.

\subsection{Acknowledgments}
I am indebted to Nir Gadish for suggesting the main idea of this paper and for further helpful conversations. I am also grateful to Geoffrey Powell and Christine Vespa for their interest in this project and their suggestions and comments, and to my PhD advisor Dan Petersen for his guidance and helpful conversations. Finally I also want to thank an anonymous referee for their detailed comments and helpful suggestions. The author was supported by Dan Petersen's grant ERC-2017-STG 759082, and also received financial support from the Knut and Alice Wallenberg Foundation.

\section{Categories of (outer) functors}\label{sec:recollection-outer-functors}

We start by introducing the relevant categories and some of their properties. The interested reader is encouraged to consult \cite{PV18} and the references therein for further details.

\begin{defi}
    Let $\Catgr$ denote the category of finitely-generated free groups and all group morphisms. A skeleton is given by the collection of free products $\FreeGroup[n]$ for all integers $n\geq 0$, with $\FreeGroup[0] = \{\ast\}$ by convention. We define $\CatFuncGr$ to be the category of functors $\Catgr\to \Vect{\Q}$, with $\Vect{\Q}$ denoting the category of rational vector spaces.
\end{defi}

\begin{expl}
    A fundamental example of a functor in $\CatFuncGr$ is the linearization functor $\ab(G) = G_{ab}\otimes\Q = H_1(G,\Q)$, acting on objects as $\FreeGroup[n]\mapsto\Q^n$. We obtain many more functors in $\CatFuncGr$ by post-composition with any functor $F\colon \Vect{\Q}\to\Vect{\Q}$. In particular, post-composition with $T^d(V) = V^{\otimes d}$ produces the functors $\TensorPower{\ab}{d}$ discussed in the introduction.
\end{expl}

The free product $\FreeGroup[m]\star\FreeGroup[n] = \FreeGroup[(m+n)]$ and the trivial group $\FreeGroup[0]$ provide a structure of symmetric monoidal category $(\Catgr, \star, \FreeGroup[0])$ generated by $\Z$, such that $\FreeGroup[0]$ is both a zero object and the unit of $\star$. We can therefore, as in \cite{HPV15}, define for each $F\in \CatFuncGr$ and each $n\in\N$ a \emph{cross-effect} $\operatorname{cr}_n F\colon\Catgr^{\times n}\to \Vect{\Q}$.
Once we have a notion of cross-effects, we can as in the original work of Eilenberg--MacLane \cite{EM54} define \emph{polynomial functors}.

\begin{defi}[Polynomial functors]\label{def:polynomial-functors}
    A functor $F\in\CatFuncGr$ is called a \emph{polynomial functor} of degree (at most) $n$ if its $(n+1)$-st cross-effect is trivial, ie $\operatorname{cr}_{n+1}F\equiv 0$.

    Let $\CatPolyFunc{n}$ be the full subcategory of $\CatFuncGr$ whose objects are the polynomial functors of degree at most $n$. The properties of cross-effect functors imply that we have a sequence of inclusions
    \begin{equation*}
        \CatPolyFunc{0} \subset \ldots \subset \CatPolyFunc{d}\subset \CatPolyFunc{d+1}\subset \ldots\subset\CatFuncGr.
    \end{equation*}

    The union $\bigcup_{d\in\N}{\CatPolyFunc{d}}$ is the category of polynomial functors, denoted $\CatPolyFunc{\mathrm{pol}}$.
\end{defi}

In the following, we will denote by $\operatorname{cr}_n^{\Z}$ the functor
\begin{equation}
\begin{aligned}
    \operatorname{cr}_n^{\Z}\colon \CatPolyFunc{n}&\to \Q[\SymGroup]\textrm{-mod}\\
    F &\mapsto \operatorname{cr}_n F(\Z,\ldots,\Z),
\end{aligned}
\end{equation}
with the action of $\SymGroup$ induced by the permutation action on $\Catgr^{\times n}$. In particular, this means that $\operatorname{cr}_n^{\Z}F$ is a \emph{left} $\Q[\SymGroup[n]]$-module. Note that our choice of notation deviates slightly from the literature \cite{DV15, PV18}.

The categories of polynomial functors $\CatPolyFunc{d}$ enjoy the following properties, which can be summarized by the \emph{recollement diagram} of \cite[Théorème 3.2]{DV15}:

\begin{prop}\label{prop:recollement-diagram-poly}\leavevmode
    \begin{enumerate}
        \item The inclusion $\operatorname{incl}_d \colon\CatPolyFunc{d}\into\CatFuncGr$ admits a left and a right adjoint $\LeftAdjq{d}, \mathbf{p}_d\colon \CatFuncGr\to\CatPolyFunc{d}$, with the additional property that both the counit $\LeftAdjq{d}\circ\operatorname{incl}_d\to \Id$ and the unit $\Id\to\mathbf{p}_d\circ\operatorname{incl}_d$ are isomorphisms.
        \item The functor $\operatorname{cr}_d^{\Z}\colon \CatPolyFunc{d}\to\Q[\SymGroup[d]]\mathrm{-mod}$ is naturally equivalent to $\Hom_{\CatPolyFunc{d}}(\TensorPower{\ab}{d}, -)$. It is exact and admits both a left and a right adjoint.
        \item The left adjoint $\alpha_d$ to $\operatorname{cr}_d^{\Z}$ is given by $\alpha_d\colon M\mapsto \TensorPower{\ab}{d} \otimes_{\SymGroup[d]} M$. This functor is exact and preserves projectives. Moreover, the unit $\Id \to \operatorname{cr}_d^{\Z}\circ\alpha_d$ is an isomorphism.
        \item The right adjoint $\beta_d$ to $\operatorname{cr}_d^{\Z}$ is exact and preserves injectives. Moreover, the counit $\operatorname{cr}_d^{\Z}\circ\beta_d\to \Id$ is an isomorphism.
    \end{enumerate}
\end{prop}

The proofs of these properties can be found in \cite[Théorème 3.2]{DV15} and in \cite[Theorem 4.4]{PV18}. 
The right adjoint $\beta_d$ is harder to describe than the left adjoint $\alpha_d$, though such a description can be found in \cite[\S 7]{PV18}.
In order to simplify the notation, we will usually omit the functor $\operatorname{incl}_d$, as well as the index $d$ in the functors $\alpha_d$ and $\beta_d$, when doing so does not cause confusion.

Recall that the irreducible $\Q\SymGroup[d]$-modules are indexed by the integer partitions $\lambda \vdash d$. We will use $\Specht{\lambda}$ to denote the irreducible module associated to $\lambda$, and we define $|\lambda| := d$. For example $\Specht{(d)}$ denotes the trivial representation of $\SymGroup[d]$, and $\Specht{(1^d)}$ denotes the sign representation.

\begin{cor}[\cite{PV18} Corollary 4.5, Corollary 4.15, Theorem 4.22]
    The simple polynomial functors of degree $d$ are the functors of the form $\alpha\Specht{\lambda}$.

    For every representation $M\in \Q[\SymGroup[d]]\mathrm{-mod}$, the functor $\beta M$ is injective (even considered as an object of $\CatPolyFunc{\mathrm{pol}}$). In fact, $\beta M$ is the injective envelope of $\alpha M$ inside $\CatPolyFunc{\mathrm{pol}}$.
\end{cor}

\begin{rmk}\label{rmk:Schur-Weyl-duality}
    The isomorphism of $\Q\SymGroup[d]$-bimodules $\Q\SymGroup[d] \cong \oplus_{\lambda\vdash d} {(\Specht{\lambda}\boxtimes\Specht{\lambda})}$ induces an isomorphism
    \begin{equation}\label{Schur-Weyl-Fgr}
        \TensorPower{\ab}{d} = \alpha\Q\SymGroup[d] \cong \bigoplus_{\lambda\vdash d}{\alpha\Specht{\lambda}\otimes\Specht{\lambda}},
    \end{equation}
    and therefore by linearity of $\Ext$ we obtain
    \begin{equation}\label{Schur-Weyl-Ext}
        \Ext^k_{\CatPolyFunc{\mathrm{pol}}}(\TensorPower{\ab}{n}, \TensorPower{\ab}{m}) \cong \bigoplus_{\nu\vdash n}\bigoplus_{\lambda\vdash m}{\Ext^k_{\CatPolyFunc{\mathrm{pol}}}(\alpha\Specht{\nu}, \alpha\Specht{\lambda})\otimes (\Specht{\nu}\boxtimes\Specht{\lambda})},
    \end{equation}
    with $\Specht{\nu}\boxtimes\Specht{\lambda}$ denoting the irreducible
    $(\Q\SymGroup[n], \Q\SymGroup[m])$-bimodule\footnote{We use here the equivalence between left $\Q\SymGroup[n]$-modules and right $\Q\SymGroup[n]^{op}$-modules.} associated to $\Specht{\nu}$ and $\Specht{\lambda}$. In other words, computing $\Ext^k_{\CatPolyFunc{\mathrm{pol}}}(\TensorPower{\ab}{n}, \TensorPower{\ab}{m})$ as a $(\Q\SymGroup[n], \Q\SymGroup[m])$-bimodule is equivalent to computing each of the individual $\Ext^k_{\CatPolyFunc{\mathrm{pol}}}(\alpha\Specht{\nu}, \alpha\Specht{\lambda})$.
\end{rmk}

The functors $\LeftAdjq{d}$ define a tower
\[
\begin{tikzcd}
    & & \CatFuncGr \arrow{ld}[swap]{\LeftAdjq{d+1}} \arrow{d}[swap]{\LeftAdjq{d}} \arrow{rd}{\LeftAdjq{d-1}} \arrow{rrrd}{\LeftAdjq{0}} & & &\\
    \ldots \arrow{r}{\LeftAdjq{d+1}} & \CatPolyFunc{d+1} \arrow{r}{\LeftAdjq{d}} & \CatPolyFunc{d} \arrow{r}{\LeftAdjq{d-1}} & \CatPolyFunc{d-1} \arrow{r} & \ldots\arrow{r} & \CatPolyFunc{0}.
\end{tikzcd}
\]
Moreover, for each integer $d$ the ajunction unit $\operatorname{Id}\to \operatorname{incl}_d\circ \LeftAdjq{d}$ is surjective.

\begin{defi}[polynomial filtration]\label{defi:polynomial-filtration}
    For any functor $F\in \CatFuncGr$, the above tower induces a (descending) filtration by defining $F_d := \ker(F\to \LeftAdjq{d-1}F)$. Up to reindexing, this is precisely the \emph{polynomial filtration} described in \cite[Definition 4.10]{PV18} (though their choice of indices makes their filtration an ascending one).
\end{defi}

\begin{defi}\label{defi:composition-factors}
    For any functor $F\in\CatFuncGr$, and any partition $\lambda\vdash d$, let $F^{\lambda}$ be the multiplicity of $\Specht{\lambda}$ in $\operatorname{cr}_d^{\Z}\LeftAdjq{d}F$, so that
    \begin{equation}\label{eq:composition-factors}
        \operatorname{cr}_d^{\Z}\LeftAdjq{d}F = \bigoplus_{\lambda\vdash d}{\Specht{\lambda}\otimes F^{\lambda}}.
    \end{equation}
\end{defi}

It is more commonly the \emph{dimension} $\dim F^{\lambda}$ which is understood to be the multiplicity of $\Specht{\lambda}$; using the vector space $F^{\lambda}$ allows us to keep track of extra structure, hence the abuse of language.

\begin{expl}
    As an illustration, $\operatorname{cr}_d\TensorPower{\ab}{d}$ is the $\Q\SymGroup[d]$-bimodule $\Q\SymGroup[d]$. In this case, we have that $(\TensorPower{\ab}{d})^{\lambda} = \Specht{\lambda}$ is not just a mere vector space, but has the structure of a right $\Q\SymGroup[d]$-module.
\end{expl}

\begin{prop}[\cite{PV18}, Proposition 4.11]\label{prop:filtration-layers}
    The quotients of the filtration defined above are given by
    \begin{equation*}
        F_d/F_{d+1} = \ker(\LeftAdjq{d}F\to\LeftAdjq{d-1}F) = \alpha\operatorname{cr}_d^{\Z}\LeftAdjq{d}F = \bigoplus_{\lambda\vdash d}{\alpha\Specht{\lambda}\otimes F^{\lambda}}.
    \end{equation*}
\end{prop}

The first equality follows from an easy diagram chase, the second equality is proven in the reference, and the last equality follows easily from the definition of the functor $\alpha$.

\begin{defi}
    In the following we will use $F^{[d]} := \operatorname{cr}_d^{\Z}\LeftAdjq{d}F \cong \operatorname{cr}_d^{\Z}(F_d/F_{d+1})$.
\end{defi}
In the introduction we defined $F^{[d]} := (F_d/F_{d+1})$ for simplicity, but this caused us to commit some abuses of notation. Since the composition $\operatorname{cr}_d^{\Z}\circ\alpha_d$ is naturally isomorphic to the identity, we see from the previous proposition that these two definitions are equivalent to each other.

\subsection{Outer functors}
Powell--Vespa introduced in \cite{PV18} the notion of \emph{outer functors}, ie functors $F\in\CatFuncGr$ such that for each $\FreeGroup\in\Catgr$ the action of $\Aut(\FreeGroup)$ on $F(\FreeGroup)$ factors through $\Out(\FreeGroup)$.

\begin{defi}
    The category $\CatOutFuncGr$ is the full subcategory of $\CatFuncGr$ whose objects are the functors $F$ such that for each $\FreeGroup\in\Catgr$ the group of inner automorphisms $\operatorname{Inn}(\FreeGroup)$ acts trivially on $F(\FreeGroup)$.

    We also define the category of \emph{outer polynomial functors} of degree (at most) $d$, denoted $\CatPolyOutFunc{d}$, to be the intersection $\CatPolyFunc{d}\cap\CatOutFuncGr$. Since both categories are full subcategories of $\CatFuncGr$, the category $\CatPolyOutFunc{d}$ is a full subcategory of all of them. We define $\CatPolyOutFunc{\mathrm{pol}}$ similarly.
\end{defi}

\begin{prop}[\cite{PV18} Propositions 9.10 and 9.25]
    The inclusion $\CatOutFuncGr\to\CatFuncGr$ admits a right adjoint $\omega\colon \CatFuncGr\to \CatOutFuncGr$. This functor is left exact, preserves injectives and is left-inverse to the inclusion functor. This inclusion also has a left adjoint.
\end{prop}

The functor $\omega$ is not exact. For example, the functor $\beta\Q\SymGroup[2]$ is a non-trivial extension $0 \to \TensorPower{\ab}{2}\to \beta\Q\SymGroup[2]\to \ab\to 0$, while after applying $\omega$ we have $\omega\beta\Q\SymGroup[2]\cong \TensorPower{\ab}{2} = \omega(\TensorPower{\ab}{2})$, and $\omega(\ab) = \ab\neq 0$ (this example is a variation on \cite[Example 9.21]{PV18}). There is also a recollement diagram for outer functors; we use the following analogue of Proposition \ref{prop:recollement-diagram-poly}.

\begin{prop}[\cite{PV18} Corollary 9.20]\label{prop:recollement-diagram-poly-out}\leavevmode
    \begin{enumerate}
        \item The functors $\LeftAdjq{d}, \mathbf{p}_d$ preserve outer functors, hence they are a left (resp. right) adjoint to the inclusion $\operatorname{incl}_d\colon\CatPolyOutFunc{d}\to\CatOutFuncGr$. Moreover the counit $\LeftAdjq{d}\circ\operatorname{incl}_d\to \Id$ and the unit $\Id\to \mathbf{p}_d\circ\operatorname{incl}_d$ are still isomorphisms.
        \item The functor $\alpha_d$ takes values in $\CatPolyOutFunc{d}$, and therefore this functor is left adjoint to $\operatorname{cr}_d^{\Z}\colon\CatPolyOutFunc{d}\to \Q[\SymGroup[d]]\textrm{-mod}$.
        \item The composition $\omega\beta_d$ is right adjoint to $\operatorname{cr}_d^{\Z}$, and the counit $\operatorname{cr}^{\Z}_d\circ \omega\beta_d \to \Id$ is an isomorphism.
    \end{enumerate}
\end{prop}

As with the category of all functors, we obtain
\begin{cor}\label{cor:simple-and-injective-outer}
    The simple outer polynomial functors of degree $d$ are the functors of the form $\alpha\Specht{\lambda}$ for all partitions $\lambda\vdash d$, and for any $M\in \Q[\SymGroup[d]]\mathrm{-mod}$, the functor $\omega\beta M$ is the injective envelope of $\alpha M$ inside $\CatPolyOutFunc{\mathrm{pol}}$.
\end{cor}

\begin{prop}\label{prop:outer-functors-co-complete}
    The category $\Foutpol$ has enough injectives. Moreover, for any positive integer $d$, the category $\CatPolyOutFunc{d}$ has enough injectives and enough projectives.
\end{prop}

\begin{proof}
    It was proved in \cite[Proposition 9.22]{PV18} that $\Foutpol$ has enough injectives. A similar argument shows that it is also true of the categories $\CatPolyOutFunc{d}$. The fact that $\CatPolyOutFunc{d}$ has enough projectives was not explicitly stated in \cite{PV18}, but can easily be obtained from their results and is similar in spirit to \cite[Corollary 9.26]{PV18}: it is a standard result that $\CatPolyFunc{d}$ has enough projectives \cite[Proposition A.9]{PV18}, and then a family of projective generators of $\CatPolyOutFunc{d}$ is obtained by applying the left adjoint to the inclusion functor $\CatPolyOutFunc{d}\to \CatPolyFunc{d}$ to a family of projective generators of $\CatPolyFunc{d}$.
\end{proof}

\section{The spectral sequence for \texorpdfstring{$\CatPolyFunc{\mathrm{pol}}$}{Fpol(gr)}}\label{sec:general-functors}

In this section we will discuss a spectral sequence computing $\Ext^*_{\CatPolyFunc{\mathrm{pol}}}(\alpha \Specht{\nu}, \alpha \Specht{\lambda})$ for simple functors $\alpha \Specht{\nu}, \alpha \Specht{\lambda}$. The results obtained are essentially a reformulation of the Koszul property for $\Ext^*_{\CatPolyFunc{\mathrm{pol}}}(\TensorPower{\ab}{n},\TensorPower{\ab}{m})$, as discussed in \cite{Pow21} in the context of \emph{contravariant analytic} functors; the main purpose of this section is to familiarize the reader with the spectral sequence before getting to the more complicated case of the corresponding spectral sequence for $\CatPolyOutFunc{\mathrm{pol}}$ in the next section.

\begin{lem}\label{lem:equal-Exts}
    For any polynomial functors $F, G\in \CatPolyFunc{d}$, there exist isomorphisms
    \[
        \Ext^*_{\CatPolyFunc{\mathrm{pol}}}(F, G) \cong \Ext^*_{\CatPolyFunc{d}}(F, G) \cong \Ext^*_{\CatFuncGr}(F, G).
    \]
\end{lem}

\begin{proof}
    The second isomorphism is \cite[Théorème 1]{DPV16}. The first isomorphism is a consequence of the fact that the inclusion $\CatPolyFunc{d}\into \CatPolyFunc{\mathrm{pol}}$ is exact and preserves injectives, since its left adjoint $\LeftAdjq{d}$ is an exact functor \cite[Proposition 4.9]{PV18}.
    Therefore $G$ admits the same injective resolution in both categories.
\end{proof}

\begin{rmk}\label{rmk:vanishing-Ext-Vespa}
    In particular, this Lemma combined with \cite[Theorem 1]{Ves18} and the isomorphism \eqref{Schur-Weyl-Ext} from Remark \ref{rmk:Schur-Weyl-duality} implies that $\Ext^k_{\CatPolyFunc{\mathrm{pol}}}(\alpha \Specht{\nu}, \alpha \Specht{\lambda})$ is trivial unless $k = |\lambda| - |\nu|\geq 0$.
\end{rmk}

\begin{prop}\label{prop:description-spectral-sequence}
    Let $F, G\in \CatPolyFunc{\mathrm{pol}}$ be two polynomial functors, then there exists a spectral sequence
    \begin{equation}\label{eq:spectral-sequence-general}
        E_1^{p,q} = \bigoplus_{\rho\vdash p}{\Ext^{p+q}_{\CatPolyFunc{\mathrm{pol}}}(F, \alpha\Specht{\rho})\otimes G^{\rho}} \Rightarrow \Ext^{p+q}_{\CatPolyFunc{\mathrm{pol}}}(F, G).
    \end{equation}
\end{prop}

This is a cohomologically graded spectral sequence, hence the differential $d_r$ has degree $(r, 1-r)$. The notation $G^{\rho}$ was defined in Definition \ref{defi:composition-factors}. The first page of this spectral sequence can equivalently be written $E_1^{p,q} = \Ext^{p+q}_{\CatPolyFunc{\mathrm{pol}}}(F, \TensorPower{\ab}{p}) \otimes_{\SymGroup[p]} G^{[p]}$. 

\begin{proof}
    This is essentially the spectral sequence of a filtered complex (see for example \cite[Theorem 5.5.1]{Weibel94}), for the filtered complex $\Hom_{\CatPolyFunc{\mathrm{pol}}}(P^k, G_d)$ associated to a projective resolution $P^{\bullet}\onto F$ and the polynomial filtration on $G$ (Definition \ref{defi:polynomial-filtration}). The only problem is that in general a polynomial functor $F$ does not admit a projective resolution in $\CatPolyFunc{\mathrm{pol}}$. Instead we construct the spectral sequence in $\CatPolyFunc{n}$ for some $n$ large enough, and Lemma \ref{lem:equal-Exts} ensures that the resulting spectral sequence does not depend on the choice of $n$.

    We obtain a $E_1$-page of the form
    \begin{equation*}
        E_1^{p,q} = \Ext^{p+q}_{\CatPolyFunc{\mathrm{pol}}}(F, G_p/G_{p+1})
    \end{equation*}
    and this gives us the expression \eqref{eq:spectral-sequence-general}, using Proposition \ref{prop:filtration-layers} and the linearity properties of $\Ext$.
\end{proof}

We will use this spectral sequence in the following case:

\begin{expl}\label{expl:main-example-non-out}
    Consider the functors $F = \alpha \Specht{\nu}, G = \beta \Specht{\lambda}$. Since $\beta \Specht{\lambda}$ is the injective envelope of $\alpha \Specht{\lambda}$ in $\CatPolyFunc{\mathrm{pol}}$, we have that $\Ext^{p+q}_{\CatPolyFunc{\mathrm{pol}}}(\alpha \Specht{\nu}, \beta\Specht{\lambda})$ is trivial unless $\nu = \lambda$ and $p+q=0$. On the other hand, as pointed out in Remark \ref{rmk:vanishing-Ext-Vespa}, we have that $\Ext^{p+q}_{\CatPolyFunc{\mathrm{pol}}}(\alpha\Specht{\nu}, \alpha\Specht{\rho})$ is trivial unless $p+q = |\rho| - |\nu|$. Since $|\rho| = p$ this means that the spectral sequence is concentrated on the line $q = -|\nu|$ and hence degenerates after the first page.
\end{expl}

Continuing the above example, note that the coefficients $(\beta\Specht{\lambda})^{\rho}$ have been determined in \cite[Corollary 11.23]{PV18}.
For the following discussion we will only use that when $|\rho| = |\lambda|$, then $(\beta \Specht{\lambda})^{\rho}$ is trivial unless $\lambda = \rho$, in which case it has rank $1$.

These properties imply that if we consider all the chain complexes associated to all partitions $\nu$ and $\lambda$, they assemble into a ``triangular'' equation system for the ranks of $\Ext^{|\rho|-|\nu|}_{\CatPolyFunc{\mathrm{pol}}}(\alpha\Specht{\nu}, \alpha\Specht{\rho})$, whose coefficients are the ranks of $(\beta\Specht{\lambda})^{\rho}$, with $\rho$ ranging over all partitions $|\nu|\leq |\rho|\leq |\lambda|$. As a consequence, we can compute the ranks of $\Ext^{|\lambda|-|\nu|}_{\CatPolyFunc{\mathrm{pol}}}(\alpha\Specht{\nu}, \alpha\Specht{\lambda})$ by an induction on $m = |\lambda|$, provided we have already computed the ranks of all $(\beta\Specht{\lambda})^{\rho}$. Conversely, it is also possible to compute the ranks of each $(\beta\Specht{\lambda})^{\rho}$ from the knowledge of the ranks of each $\Ext^*_{\CatPolyFunc{\mathrm{pol}}}(\alpha\Specht{\nu}, \alpha\Specht{\lambda})$; this type of computation will be discussed in more detail in Theorem \ref{thm:determine-composition-factors}.

More conceptually, using the isomorphism \eqref{Schur-Weyl-Ext} from Remark \ref{rmk:Schur-Weyl-duality} to assemble all these chain complexes into a single chain complex of $(\Q\SymGroup[\bullet], \Q\SymGroup[\bullet])$-bimodules, we obtain an expression exhibiting the algebra constructed from\footnote{The following isomorphism is obtained from the adjunction properties of the functors involved, and the fact that finite-dimensional representations of symmetric groups over $\Q$ are self-dual.} $\Hom_{\CatPolyFunc{\mathrm{pol}}}(\beta\Q\SymGroup[m], \beta\Q\SymGroup[n])\cong (\beta\Q\SymGroup[m])^{[n]}$ as Koszul dual to the algebra 
\begin{equation*}
    \bigoplus_{m,n\in\N}{\Ext^{m-n}_{\CatPolyFunc{\mathrm{pol}}}(\TensorPower{\ab}{m}, \TensorPower{\ab}{n})},
\end{equation*}
whose product is induced from the Yoneda product. We will not develop in more detail this Koszul property here, and refer instead the interested reader to \cite[\S 8.1 and Thm 6.22]{Pow21}.

\section{The spectral sequence for \texorpdfstring{$\CatPolyOutFunc{\mathrm{pol}}$}{Foutpol(gr)}}\label{sec:outer-functors}
We now turn to the main goal of this paper, which is the study of $\Ext^*_{\CatPolyOutFunc{\mathrm{pol}}}(\alpha\Specht{\nu}, \alpha\Specht{\lambda})$. We first state the counterpart to Lemma \ref{lem:equal-Exts}, which is proved similarly.

\begin{lem}
    For any outer polynomial functors $F, G\in \CatPolyOutFunc{d}$, there exists an isomorphism
    \begin{equation*}
        \Ext^*_{\Foutpol}(F, G) \cong \Ext^*_{\CatPolyOutFunc{d}}(F, G).
    \end{equation*}
\end{lem}

\begin{rmk}
    It is worth emphasizing however that it is not known whether the comparison map $\Ext^*_{\CatPolyOutFunc{d}}(F, G)\to \Ext^*_{\CatOutFuncGr}(F, G)$ is an isomorphism. This is why we restrict ourselves to working in the category of polynomial outer functors instead of the category of all outer functors.
\end{rmk}

Using the same arguments as for Proposition \ref{prop:description-spectral-sequence}, we obtain

\begin{prop}\label{prop:spectral-sequence-outer}
    Let $F, G\in \Foutpol$ be two outer polynomial functors. There exists a spectral sequence
    \begin{equation}\label{eq:spec-seq-outer}
        E_1^{p,q} = \bigoplus_{\rho\vdash p}{\Ext^{p+q}_{\Foutpol}(F, \alpha\Specht{\rho})\otimes G^{\rho}} \Rightarrow \Ext^{p+q}_{\Foutpol}(F, G).
    \end{equation}
\end{prop}

As in the previous section, this spectral sequence is cohomologically graded, meaning that the differential $d_r$ has degree $(r, 1-r)$. The notation $G^{\rho}$ comes from Definition \ref{defi:composition-factors}, and as in the previous section we can equivalently write $E_1^{p,q} = \Ext^{p+q}_{\Foutpol}(F, \TensorPower{\ab}{p})\otimes_{\SymGroup[p]} G^{[p]}$.

For the rest of this paper we will work in the special case when $F = \alpha\Specht{\nu}$ and $G = \omega\beta\Specht{\lambda}$. We record the following fact for future reference:

\begin{lem}\label{lem:spectral-sequence-vanishes}
    If $\nu \neq \lambda$ are two distinct partitions, then $\Ext^*_{\Foutpol}(\alpha\Specht{\nu},\omega\beta\Specht{\lambda}) = 0$. If $\nu = \lambda$ then $\Ext^0_{\Foutpol}(\alpha\Specht{\lambda},\omega\beta\Specht{\lambda}) = \Q$ and $\Ext^*_{\Foutpol}(\alpha\Specht{\lambda},\omega\beta\Specht{\lambda}) = 0$ for all $*\neq 0$.
\end{lem}

\begin{proof}
    Since $\omega\beta\Specht{\lambda}$ is injective in $\Foutpol$, we immediately obtain that the above expression vanishes in all cohomological degrees $* \neq 0$. In cohomological degree $* = 0$, the result follows from the fact that $\alpha$ is left adjoint to $\operatorname{cr}_d^{\Z}$ and $\operatorname{cr}_d^{\Z}\circ \omega\beta$ is naturally isomorphic to the identity (cf Proposition \ref{prop:recollement-diagram-poly-out}).
\end{proof}

First, standard properties of polynomial functors allow us to determine certain terms on the $E_1$-page of this spectral sequence.

\begin{lem}\label{lem:special-values-E1}
    In the spectral sequence \eqref{eq:spec-seq-outer} associated to $F = \alpha\Specht{\nu}, G = \omega\beta\Specht{\lambda}$, we have an explicit description for the following terms on the $E_1$-page:
    \begin{enumerate}
        \item $E_1^{p,q} = 0$ for all $p > |\lambda|$
        \item $E_1^{p,q} = 0$ if $p < 2$ and $|\lambda|\geq 2$
        \item $E_1^{p,q} = 0$ for all $q < -p$
        \item $E_1^{|\lambda|, q} = \Ext^{|\lambda|+q}_{\Foutpol}(\alpha\Specht{\nu}, \alpha\Specht{\lambda})$.
    \end{enumerate}
\end{lem}

\begin{proof}
    In order to prove the first statement it suffices to notice that a polynomial functor of degree $d$ has no composition factor of degree larger than $d$, hence $(\omega\beta\Specht{\lambda})^{\rho} = 0$ for all $|\rho| = p > |\lambda|$.

    The second statement follows from the vanishing of $(\omega\beta\Specht{\lambda})^{\rho}$ for $|\rho|\leq 1$ when $|\lambda|\geq 2$. This vanishing is a consequence of the fact that both the constant functor $\Q = \alpha\Specht{(0)}$ and the abelianization functor $\ab=\alpha\Specht{(1)}$ are projective in $\Foutpol$ (and in fact even in $\CatPolyOutFunc{}$, see \cite[Proposition 9.32]{PV18}). If $\Q$ were to appear as a composition factor of $\omega\beta\Specht{\lambda}$, then by projectivity it would be a direct summand, contradicting the fact that $\omega\beta\Specht{\lambda}$ is the injective envelope of $\alpha\Specht{\lambda}$. We then do the same argument for the functor $\ab$.

    The third statement is simply the vanishing of $\Ext^*$ for $*<0$.

    Finally, the last statement follows from the fact that $\omega\beta\Specht{\lambda}$ is the injective envelope of $\alpha\Specht{\lambda}$. In particular it is an essential extension and the only composition factor of $\omega\beta\Specht{\lambda}$ of degree $|\lambda|$ is $\alpha\Specht{\lambda}$, hence for $|\rho| = |\lambda|$ we have that $(\omega\beta\Specht{\lambda})^{\rho}$ has rank $1$ if $\rho=\lambda$ and otherwise vanishes.
\end{proof}

Combining these observations, we obtain:

\begin{prop}\label{prop:vanishing-Ext-Out-high-degree}
    The groups $\Ext^k_{\Foutpol}(\TensorPower{\ab}{n},\TensorPower{\ab}{m})$ vanish for all $k>m-n$.
\end{prop}

\begin{proof}
    Using the isomorphism \eqref{Schur-Weyl-Ext} from Remark \ref{rmk:Schur-Weyl-duality}, it is enough to prove that for all partitions $\nu\vdash n, \lambda\vdash m$ and all $k > m-n$, we have $\Ext^k_{\Foutpol}(\alpha\Specht{\nu},\alpha\Specht{\lambda}) = 0$.
    
    We proceed by induction on $m$. If $m=0$ or $m=1$, the only partition is $\lambda = (m)$ and it follows from \cite[Corollary 12.10]{PV18} that $\alpha\Specht{(m)}$ is injective. This proves vanishing in degree $k\neq 0$, and in degree $k=0$ we must have $m<n$ by assumption, so the result follows from Schur's lemma.

    Otherwise for $m\geq 2$ we have by Lemma \ref{lem:special-values-E1}(4.) that $\Ext^k_{\Foutpol}(\alpha\Specht{\nu},\alpha\Specht{\lambda}) = E_1^{m,k-m}$. Moreover, by Lemma \ref{lem:spectral-sequence-vanishes} $E_{\infty}^{m,k-m} = 0$ since $\Ext^k_{\Foutpol}(\alpha\Specht{\nu}, \omega\beta\Specht{\lambda})= 0$. In order to finish the argument, it is enough to show that there can be no non-trivial differential at $E_r^{m,k-m}$ for any $r\geq 1$.
    
    Clearly there can be no non-trivial outgoing differential, as $E_1^{p,q} = 0$ for all $p>m$ by Lemma \ref{lem:special-values-E1}(1.). There can also be no non-trivial ingoing differential, as any $d_r$ differential into $E_r^{m,k-m}$ would have to start from $E_r^{m-r, k+r-m-1}$, and by the induction hypothesis the corresponding term on the $E_1$-page is already trivial: indeed we have $k-1 \geq k-r > m-r-n$, and therefore $\Ext^{k-1}_{\Foutpol}(\alpha\Specht{\nu},\alpha\Specht{\rho}) = 0$ for each $\rho\vdash m-r$. From the description of the spectral sequence \eqref{eq:spec-seq-outer}, this proves that indeed $E_1^{m-r, k+r-m-1} = 0$, and hence concludes the proof.
\end{proof}

Prior to this work, the only known values for $\Ext^*_{\Foutpol}$ where the following ones:

\begin{prop}\label{prop:values-Ext0-1}
    For partitions $\nu, \lambda$,
    \begin{enumerate}
        \item $\Ext^0_{\Foutpol}(\alpha\Specht{\nu},\alpha\Specht{\lambda}) = \left\{\begin{array}{cc}
            \Q & \text{ if $\nu = \lambda$} \\
            0 & \text{ if $\nu\neq \lambda$}
        \end{array}\right.$
        \item If $|\nu|+1\neq |\lambda|$ then $\Ext^1_{\Foutpol}(\alpha\Specht{\nu}, \alpha\Specht{\lambda}) = 0$, and when $|\nu|+1 = |\lambda|$ there is an equality
        \begin{equation*}
            \Ext^1_{\Foutpol}(\alpha\Specht{\nu}, \alpha\Specht{\lambda}) = (\omega\beta\Specht{\lambda})^{\nu}.
        \end{equation*}
    \end{enumerate}
\end{prop}

The first statement is standard and the second one was already proved in \cite[Corollary 12.13]{PV18}.
We provide a proof anyway for the reader's convenience, and because the argument for $\Ext^1$ illustrates how we can use the spectral sequence.

\begin{proof}
    The first statement is an immediate consequence of Schur's lemma. For the second statement, the fact that $\Foutpol$ is a full subcategory of $\CatFuncGr$ means that we have an injection
    \begin{equation*}
        0\to \Ext^1_{\Foutpol}(\alpha\Specht{\nu}, \alpha\Specht{\lambda})\to \Ext^1_{\CatFuncGr}(\alpha\Specht{\nu}, \alpha\Specht{\lambda}),
    \end{equation*}
    using that $\Ext^1$ classifies extensions of $\alpha\Specht{\nu}$ by $\alpha\Specht{\lambda}$ (see eg \cite[Theorem 3.4.3]{Weibel94}). Combining this injection with the vanishing from \cite[Theorem 1]{Ves18} provides the vanishing for $|\nu|+1\neq|\lambda|$. 
    
    In the case $|\nu|+1 = |\lambda|$, the only non-trivial terms in the associated spectral sequence are $E_1^{|\nu|,-|\nu|} = (\omega\beta\Specht{\lambda})^{\nu}$ (using the result just proven for $\Ext^0$), and $E_1^{|\lambda|,-|\nu|} = \Ext^1_{\Foutpol}(\alpha\Specht{\nu},\alpha\Specht{\lambda})$. Since the spectral sequence converges to $0$ by Lemma \ref{lem:spectral-sequence-vanishes}, the $d_1$-differential between these two terms has to be an isomorphism.
\end{proof}

Combining the above results, we can complement the information obtained from Lemma \ref{lem:special-values-E1} as follows:

\begin{cor}\label{cor:more-special-values-E1}
    In the spectral sequence \eqref{eq:spec-seq-outer} associated to $F = \alpha\Specht{\nu}, G = \omega\beta\Specht{\lambda}$, the following terms vanish:
    \begin{enumerate}
        \item $E_1^{p,q} = 0$ for all $q > -|\nu|$
        \item $E_1^{p,-p} = 0$ for all $p\neq |\nu|$
        \item $E_1^{p+1,-p} = 0$ for all $p\neq |\nu|$.
    \end{enumerate}
\end{cor}

\begin{proof}
    The first statement is a direct corollary of Proposition \ref{prop:vanishing-Ext-Out-high-degree}, while the other two statements are a direct corollary of Proposition \ref{prop:values-Ext0-1}.
\end{proof}

As a (partial) summary of these vanishing statements, we state the following

\begin{prop}\label{prop:triangle-E1}
    Consider the spectral sequence \eqref{eq:spec-seq-outer} associated to $F = \alpha\Specht{\nu}, G = \omega\beta\Specht{\lambda}$. If $|\nu| > |\lambda|$, then the $E_1$-page is identically zero (and therefore so is the whole spectral sequence). Otherwise, the nonzero terms are concentrated in the triangle with corners $(p,q) = (|\nu|, -|\nu|), (|\lambda|,-|\nu|), (|\lambda|, -|\lambda|)$. Moreover, for all $\nu < \lambda$ the differential $d_1\colon E_1^{|\nu|,-|\nu|}\to E_1^{|\nu|+1,-|\nu|}$ is injective.
\end{prop}

\begin{proof}
    The triangular shape is obtained by combining the first statement of Corollary \ref{cor:more-special-values-E1} with Lemma \ref{lem:special-values-E1} (1. and 3.). The injectivity follows from the fact that $E_{\infty}^{|\nu|,-|\nu|} = 0$ by Lemma \ref{lem:spectral-sequence-vanishes}, and the indicated differential is the only one that may be non-trivial.
\end{proof}

These statements are summarized in Figure \ref{fig:E1-page}, depicting the $E_1$-page of the spectral sequence associated to $F = \alpha\Specht{\nu}, G = \omega\beta\Specht{\lambda}$, here in the case $\nu\vdash 4$ and $\lambda\vdash 9$, with $p$ as the horizontal coordinate and $q$ as the vertical coordinate. The grey triangle is the region where the terms may be non-zero. This region contains exactly $6 = |\lambda| - |\nu| + 1$ columns. In that region, the empty circles represent terms that are actually $0$ due to the vanishing results for $\Ext^0$ and $\Ext^1$. The crosses on the top row are the ``Koszul region''; for the extensions inside $\CatPolyFunc{\mathrm{pol}}$ discussed in the previous section they are the only nonzero terms. The plain circles are the remaining terms, which have in general no reason to vanish according to our current knowledge (and we prove below that some of them are indeed non-trivial). The rightmost terms are $\Ext^{p+q}_{\Foutpol}(\alpha\Specht{\nu},\alpha\Specht{\lambda})$, as indicated by their label.

\begin{figure}[h]
    \centering
    \includegraphics[scale = 0.75]{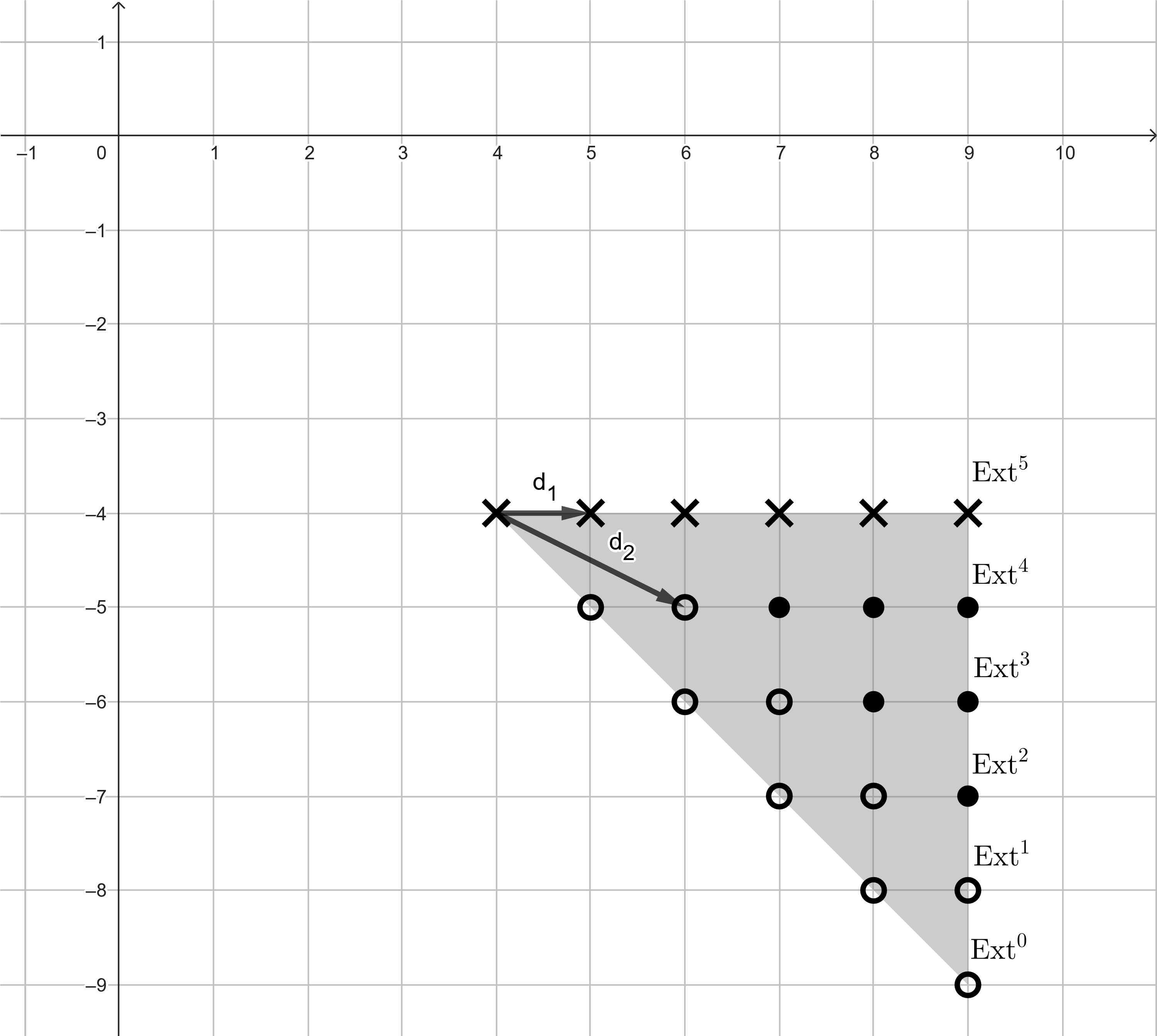}
    \caption{Depiction of the $E_1$-page in the case $\nu\vdash 4, \lambda\vdash 9$}
    \label{fig:E1-page}
\end{figure}

\subsection{Main results}
We list below the main results of this paper. Most of them are variations or special cases of Theorem \ref{thm:determine-composition-factors}. This theorem is expressed in terms of the \emph{Euler characteristic}\footnote{or \emph{Euler--Poincaré characteristic}}, $\chi(\Ext^*_{\Foutpol}(\alpha\Specht{\nu},\alpha\Specht{\lambda})) := \sum{(-1)^i \dim(\Ext^i_{\Foutpol}(\alpha\Specht{\nu}, \alpha\Specht{\lambda}))}$.

In preparation for the next theorem, we state and prove the following equations, relating the values of $\chi(\Ext^*_{\Foutpol}(\alpha\Specht{\nu},\alpha\Specht{\rho}))$ and the ranks $\dim(\omega\beta\Specht{\lambda})^{\rho}$.

\begin{lem}\label{lem:triangular-equalities-spectral-sequence}
    Given any two partitions $\nu, \lambda$, we have the following relations, with $\delta_{\nu,\lambda}$ denoting a \emph{Kronecker delta}, which takes the value $1$ if $\nu = \lambda$, and $0$ otherwise
    \begin{enumerate}
        \item If $|\nu| > |\lambda|$, then $\chi(\Ext^*_{\Foutpol}(\alpha\Specht{\nu}, \alpha\Specht{\lambda})) = 0 = \dim(\omega\beta\Specht{\lambda})^{\nu}$.
        \item If $|\nu| = |\lambda|$, then $\chi(\Ext^*_{\Foutpol}(\alpha\Specht{\nu}, \alpha\Specht{\lambda})) = \delta_{\nu,\lambda} = \dim(\omega\beta\Specht{\lambda})^{\nu}$.
        \item If $|\nu| < |\lambda|$, the following equation is satisfied
        \begin{equation}\label{eq:triangular-system}
            0 = \sum_{p=|\nu|}^{|\lambda|}{\sum_{\rho\vdash p} \dim(\omega\beta\Specht{\lambda})^{\rho}\cdot \chi(\Ext^*_{\Foutpol}(\alpha\Specht{\nu},\alpha\Specht{\rho}))}.
        \end{equation}
        Moreover, for $p = |\nu|$ the only non-trivial term in the inner sum is $\dim(\omega\beta\Specht{\lambda})^{\nu}$, and for $p = |\lambda|$ the only non-trivial term in the inner sum is $\chi(\Ext^*_{\Foutpol}(\alpha\Specht{\nu},\alpha\Specht{\lambda}))$.
    \end{enumerate}
\end{lem}

\begin{proof}
    The first statement is a corollary of Proposition \ref{prop:vanishing-Ext-Out-high-degree} for the first equality, and of the fact that $\omega\beta\Specht{\lambda}$ belongs to $\CatPolyOutFunc{|\lambda|}$ for the second equality.

    For the second statement, the first equality also follows from Proposition \ref{prop:vanishing-Ext-Out-high-degree}, combined with Schur's lemma for computing $\Ext^0_{\Foutpol}(\alpha\Specht{\nu},\alpha\Specht{\lambda})$. The second equality follows from the recollement diagram of Proposition \ref{prop:recollement-diagram-poly-out}, and more specifically from the isomorphism $\operatorname{cr}^{\Z}_{|\lambda|}\omega\beta\Specht{\lambda} \cong \Specht{\lambda}$.

    Finally for the third statement we consider the spectral sequence \eqref{eq:spec-seq-outer} associated to $F = \alpha\Specht{\nu}$ and $G = \omega\beta\Specht{\lambda}$. Since $\nu\neq \lambda$. By Lemma \ref{lem:spectral-sequence-vanishes}, this spectral sequence converges to $0$, and therefore on each page the Euler characteristic must be $0$. Thus we obtain
    \begin{align*}
        0 = &\sum_{p,q}{(-1)^{p+q}\dim(E_1^{p,q})} \\
         = &\sum_p \sum_{\rho\vdash p} {\Big(\sum_q{(-1)^{p+q}\dim(\omega\beta\Specht{\lambda})^{\rho}}\cdot \dim(\Ext^{p+q}_{\Foutpol}(\alpha\Specht{\nu},\alpha\Specht{\rho}))}\Big)  \\
         = &\sum_{p=|\nu|}^{|\lambda|} \sum_{\rho\vdash p} {\dim(\omega\beta\Specht{\lambda})^{\rho}\cdot\chi(\Ext^{*}_{\Foutpol}(\alpha\Specht{\nu},\alpha\Specht{\rho}))}.
    \end{align*}
    The equality between the last two lines uses the first statement of this lemma. The second statement of this lemma implies that the inner sums simplify as indicated when $p = |\nu|$ and $p = |\lambda|$.
\end{proof}

The significance of these equations is highlighted by the next theorem

\begin{thm}\label{thm:determine-composition-factors}
    The collection of all ranks $\{\dim(\omega\beta\Specht{\lambda})^{\nu}\mid \nu,\lambda\}$ and the collection of all Euler characteristics $\{\chi(\Ext^*_{\Foutpol}(\alpha\Specht{\nu}, \alpha\Specht{\lambda}))\mid \nu,\lambda\}$ uniquely determine each other.
    
    In fact, for a given pair of partitions $\nu, \lambda$, we can compute $\chi(\Ext^*_{\Foutpol}(\alpha\Specht{\nu},\alpha\Specht{\lambda}))$ provided we already know $\dim(\omega\beta\Specht{\rho})^{\rho'}$ for all $|\nu|\leq |\rho'|\leq |\rho|\leq |\lambda|$. Conversely, we can compute $\dim(\omega\beta\Specht{\lambda})^{\nu}$ provided we already know $\chi(\Ext^*_{\Foutpol}(\alpha\Specht{\rho'}, \alpha\Specht{\rho}))$ for all $|\nu|\leq |\rho'|\leq |\rho|\leq |\lambda|$.
\end{thm}

\begin{proof}
    Before starting with the details of the proof, let us explain the general idea: the equations from the previous lemma \ref{lem:triangular-equalities-spectral-sequence} can be assembled into an (infinite) triangular system of equations, either with variables $\chi(\Ext^*_{\Foutpol}(\alpha\Specht{\nu},\alpha\Specht{\lambda}))$ and parameters $\dim(\omega\beta\Specht{\rho})^{\rho'}$, or conversely. The matrix corresponding to this system is invertible, which proves our claim.

    More explicitly, we will explain how to compute all the $\chi(\Ext^*_{\Foutpol}(\alpha\Specht{\nu},\alpha\Specht{\lambda}))$, provided we already know the ranks $\dim(\omega\beta\Specht{\rho})^{\rho'}$. The argument for the other direction is completely analogous. We fix a partition $\nu$, and we recursively compute $\chi(\Ext^*_{\Foutpol}(\alpha\Specht{\nu},\alpha\Specht{\lambda}))$ by a recursion on $|\lambda|$.

    By Lemma \ref{lem:triangular-equalities-spectral-sequence} we already know $\chi(\Ext^*_{\Foutpol}(\alpha\Specht{\nu},\alpha\Specht{\lambda}))$ when $|\lambda| \leq |\nu|$, so that gives us the base case of the recursion. Now assume that we have already computed the Euler characteristics $\chi(\Ext^*_{\Foutpol}(\alpha\Specht{\nu},\alpha\Specht{\rho}))$ for all partitions $\rho$ with $|\rho| < |\lambda|$. Since in Equation \eqref{eq:triangular-system} the inner sum when $p = |\lambda|$ simplifies to $\chi(\Ext^*_{\Foutpol}(\alpha\Specht{\nu},\alpha\Specht{\lambda}))$, we can rearrange the equation in order to express $\chi(\Ext^*_{\Foutpol}(\alpha\Specht{\nu},\alpha\Specht{\lambda}))$ in terms of the ranks $\dim(\omega\beta\Specht{\lambda})^{\rho}$, which we assume are already given, and of the Euler characteristics $\chi(\Ext^*_{\Foutpol}(\alpha\Specht{\nu},\alpha\Specht{\rho}))$ with $|\rho| < |\lambda|$, which we have already computed at the previous steps of the recursion. Therefore this finishes the recursive step.

    We saw in the last paragraph that $\chi(\Ext^*_{\Foutpol}(\alpha\Specht{\nu},\alpha\Specht{\lambda}))$ is computed from $\dim(\omega\beta\Specht{\lambda})^{\rho}$ and $\chi(\Ext^*_{\Foutpol}(\alpha\Specht{\nu},\alpha\Specht{\rho}))$ for $|\nu|\leq |\rho| < |\lambda|$. By induction, these latter Euler characteristics depend only on $\dim(\omega\beta\Specht{\rho})^{\rho'}$ for $|\nu| \leq |\rho'|\leq |\rho| < |\lambda|$, and therefore we obtain indeed that $\chi(\Ext^*_{\Foutpol}(\alpha\Specht{\nu},\alpha\Specht{\lambda}))$ only depends on $\dim(\omega\beta\Specht{\rho})^{\rho'}$ for $|\nu| \leq |\rho'|\leq |\rho| \leq |\lambda|$.
\end{proof}

The second part of Theorem \ref{thm:determine-composition-factors} above is very important from a practical perspective, since it means that we can compute the Euler characteristics $\chi(\Ext^*_{\Foutpol}(\alpha\Specht{\nu}, \alpha\Specht{\lambda}))$ even if we don't know all the ranks $\dim(\omega\beta\Specht{\lambda})^{\nu}$; we only need to know enough of them. More concretely, Gadish and the author computed in \cite{GH22} these ranks in the range $|\lambda|\leq 10$, and by the previous theorem this is enough to compute the Euler characteristics $\chi(\Ext^*_{\Foutpol}(\alpha\Specht{\nu}, \alpha\Specht{\lambda}))$ in the same range ($|\lambda|\leq 10)$. A direct consequence of these computations is our next theorem. Before stating it, let us briefly explain how to translate between the terminology of \cite{GH22} and the terminology used in the present paper:

\begin{rmk}\label{rmk:translation-from-GH}
    Combining the discussion of \cite[\S 4.4, \S 5.1]{GH22} and \cite[Theorem 16.11]{PV18}, we obtain that $(\omega\beta\Specht{\lambda})^{\nu} = \Phi^{|\lambda|-|\nu|}[\TransposePartition{\lambda}, \TransposePartition{\nu}]$, with $\TransposePartition{\lambda}$ denoting the partition conjugate to $\lambda$ and similarly $\TransposePartition{\nu}$ being the partition conjugate to $\nu$.
    
    The notation $\Phi$ is not used in the present paper.
\end{rmk}

\begin{thm}
    There exist triples $(k,n,m)$ with $k<m-n$ such that $\Ext^k_{\Foutpol}(\TensorPower{\ab}{n}, \TensorPower{\ab}{m})\neq 0$. For example $\Ext^4_{\Foutpol}(\TensorPower{\ab}{4}, \TensorPower{\ab}{9})\neq 0$, and also $\Ext^k_{\Foutpol}(\TensorPower{\ab}{3},\TensorPower{\ab}{10})$ is non-trivial in some degree $1 < k < 7$.
\end{thm}

\begin{proof}
    This is essentially a direct corollary of Theorem \ref{thm:determine-composition-factors} and Remark \ref{rmk:translation-from-GH}. Using the computations performed in \cite{GH22}, we can implement the recursive algorithm described in the proof of Theorem \ref{thm:determine-composition-factors} and compute $\chi(\Ext^*_{\Foutpol}(\alpha\Specht{\nu},\alpha\Specht{\lambda}))$ for all $|\lambda| \leq 10$ (recall from Corollary \ref{cor:more-special-values-E1} that the only non-trivial cases are when $|\nu| \leq |\lambda|)$. If we have $(-1)^{|\lambda| - |\nu|}\chi(\Ext^*_{\Foutpol}(\alpha\Specht{\nu},\alpha\Specht{\lambda})) < 0$, then certainly $\Ext^*_{\Foutpol}(\alpha\Specht{\nu},\alpha\Specht{\lambda})$ cannot be concentrated in degree $* = |\lambda| - |\nu|$. In the range considered here this happens in exactly two cases: $\nu = (4)$ and $\lambda = (1^9)$, as well as $\nu = (1^3)$ and $\lambda = (3^2, 2^2)$. The algorithms and some computations are accessible here\footnote{https://github.com/louishainaut/Ext-Outer-Functors}.

    As an illustration, for the case $(\nu, \lambda) = ((4), (1^9))$, \cite[Theorem 12.23]{PV18} showed that the composition factors of $\omega\beta\Specht{\lambda}$ are $\alpha\Specht{(1^9)}, \alpha\Specht{(2, 1^6)}, \alpha\Specht{(3,1^4)}, \alpha\Specht{(4, 1^2)}, \alpha\Specht{(5)}$, and the previous steps of the recursion provide 
    \begin{align*}
        \chi(\Ext^*_{\Foutpol}(\alpha\Specht{(4)},\alpha\Specht{(5)})) = 0\qquad
        \chi(\Ext^*_{\Foutpol}(\alpha\Specht{(4)},\alpha\Specht{(4, 1^2)})) = 0\\
        \chi(\Ext^*_{\Foutpol}(\alpha\Specht{(4)},\alpha\Specht{(3, 1^4)})) = -2\qquad
        \chi(\Ext^*_{\Foutpol}(\alpha\Specht{(4)},\alpha\Specht{(2, 1^6)})) = 1.
    \end{align*}
    We obtain therefore $\chi(\Ext^*_{\Foutpol}(\alpha\Specht{(4)},\alpha\Specht{(1^9)})) = 1$.

    In certain favorable cases the spectral sequences allow us to obtain information beyond the Euler characteristic. For example in the case $\nu = (4)$ and $\lambda = (1^9)$, the Euler characteristic computation implies that at least one of $\Ext^4_{\Foutpol}(\alpha\Specht{(4)},\alpha\Specht{(1^9)})$ and $\Ext^2_{\Foutpol}(\alpha\Specht{(4)},\alpha\Specht{(1^9)})$ must be non-trivial. If $\Ext^2_{\Foutpol}(\alpha\Specht{(4)},\alpha\Specht{(1^9)}) = E_1^{9,-7}$ were non-trivial, then since that term does not survive to the $E_{\infty}$-page there must be some non-trivial differential with that term as either the source or the target. However it follows from Corollary \ref{cor:more-special-values-E1} (see also Figure \ref{fig:E1-page}) that the only possible such differential would be $d_4\colon E_4^{5,-4}\to E_4^{9,-7}$, and this differential must be trivial since we have seen above that $E_1^{5,-4} = \Ext^1_{\Foutpol}(\alpha\Specht{(4)},\alpha\Specht{(5)})$ is already trivial. Therefore the only remaining possibility is that $\Ext^4_{\Foutpol}(\alpha\Specht{(4)},\alpha\Specht{(1^9)})$ is non-trivial.
\end{proof}

\begin{rmk}
    It is possible that there exist more pairs $(\nu, \lambda)$ in the range $|\lambda|\leq 10$ for which $\Ext^*_{\Foutpol}(\alpha\Specht{\nu},\alpha\Specht{\lambda})$ is not concentrated in degree $* = |\lambda| - |\nu|$, but these two are the only ones that can be detected from the Euler characteristic.
\end{rmk}

Recall from the introduction that in the category $\CatFuncGr$ it is true for all partitions $\nu,\lambda$ that $\Ext^*_{\CatFuncGr}(\alpha\Specht{\nu}, \alpha\Specht{\lambda})$ is concentrated in degree $* = |\lambda| - |\nu|$ (see \cite{Ves18}), and in \cite{Pow21} this property was interpreted as a certain Koszul-type property of the category $\CatFuncGr$. The previous theorem says that $\Foutpol$ does not have this Koszul property, but we define by analogy

\begin{terminology}
    For a given a pair of partitions $(\nu, \lambda)$, we will say that $\Ext^*_{\Foutpol}(\alpha\Specht{\nu},\alpha\Specht{\lambda})$ has the \emph{Koszul property} if it is trivial in all degrees $* \neq |\lambda| - |\mu|$.
\end{terminology}

\begin{rmk}
    At the moment, it is generally unknown whether $\Ext^*_{\Foutpol}(\alpha\Specht{\nu},\alpha\Specht{\lambda})$ has the Koszul property. However in some cases we can prove that the Koszul property is satisfied; in those cases we can compute $\Ext^{|\lambda| - |\nu|}_{\Foutpol}(\alpha\Specht{\nu},\alpha\Specht{\lambda})$ from the identity $\chi(\Ext^*_{\Foutpol}(\alpha\Specht{\nu},\alpha\Specht{\lambda})) = (-1)^{|\lambda| - |\nu|}\dim(\Ext^{|\lambda| - |\nu|}_{\Foutpol}(\alpha\Specht{\nu},\alpha\Specht{\lambda}))$.
\end{rmk}

As an illustration of this last remark, we have the following corollary to Theorem \ref{thm:determine-composition-factors}

\begin{cor}\label{cor:computing-Ext2}
    In all degrees $k\neq 2$, we have $\Ext^k_{\Foutpol}(\TensorPower{\ab}{n}, \TensorPower{\ab}{n+2}) = 0$. Moreover, $\Ext^2_{\Foutpol}(\TensorPower{\ab}{n}, \TensorPower{\ab}{n+2})$ can be explicitly computed by the following short exact sequence of $(\Q\SymGroup[n], \Q\SymGroup[n+2])$-bimodules:
    \begin{equation}\label{eq:formula-Ext2}
        (\omega\beta\Q\SymGroup[n+2])^{[n]}\into \Ext^1_{\Foutpol}(\TensorPower{\ab}{n},\TensorPower{\ab}{n+1})\otimes_{\SymGroup[n+1]}(\omega\beta\Q\SymGroup[n+2])^{[n+1]}\onto \Ext^2_{\Foutpol}(\TensorPower{\ab}{n}, \TensorPower{\ab}{n+2}).
    \end{equation}
\end{cor}

\begin{proof}
    The vanishing in degrees $k\neq 2$ is obtained by combining Propositions \ref{prop:vanishing-Ext-Out-high-degree} and \ref{prop:values-Ext0-1}.

    Since $\Ext^*_{\Foutpol}(\TensorPower{\ab}{n}, \TensorPower{\ab}{n+2})$ has the Koszul property we can use Theorem \ref{thm:determine-composition-factors} to compute $\Ext^2_{\Foutpol}(\TensorPower{\ab}{n}, \TensorPower{\ab}{n+2})$, and we obtain precisely the formula indicated.
\end{proof}

\begin{expl}
    In the short exact sequence \eqref{eq:formula-Ext2}, the middle term can be computed from Proposition \ref{prop:values-Ext0-1}, and a closed form formula can be obtained from \cite[Corollary 12.13]{PV18}. The leftmost term has been computed in \cite{GH22}\footnote{The computations can be accessed on the author's personal website \url{https://louishainaut.github.io/GH-ConfSpace/}} up to $n=8$, and therefore we can compute $\Ext^2_{\Foutpol}(\TensorPower{\ab}{n},\TensorPower{\ab}{n+2})$ in the same range. Assuming \cite[Conjecture 6.7]{GH22}, we can compute $(\omega\beta\Q\SymGroup[n])^{[n+2]}$ for all $n$ , and from there we can compute $\Ext^2_{\Foutpol}(\TensorPower{\ab}{n},\TensorPower{\ab}{n+2})$ for $n$ as large as we want. This conjecture has now been proved in recent work of Powell \cite{PowCoadjoint}.
\end{expl}

The ranks $\dim(\Ext^2_{\Foutpol}(\TensorPower{\ab}{n},\TensorPower{\ab}{n+2}))$ are given in Table \ref{tab:rank_Ext2}. We will discuss further the structure of $\Ext^2_{\Foutpol}(\TensorPower{\ab}{n},\TensorPower{\ab}{n+2})$ as a $(\Q\SymGroup[n],\Q\SymGroup[n+2])$-bimodule at the end of this paper.

\begin{table}[h!]
    \centering
    \begin{tabular}{c|c|c|c|c|c|c|c|c|c}
        $n$ & $0$ & $1$ & $2$ & $3$ & $4$ & $5$ & $6$ & $7$ & $8$ \\
        \hline
        $\dim$ & $0$ & $0$ & $0$ & $28$ & $478$ & $6718$ & $90718$ & $1239838$ & $17539198$
    \end{tabular}
    \caption{Dimension of $\operatorname{Ext}^{2}_{\Foutpol}(\TensorPower{\ab}{n}, \TensorPower{\ab}{n+2})$}
    \label{tab:rank_Ext2}
\end{table}

\begin{prop}
    The groups $\Ext^k_{\Foutpol}(\TensorPower{\ab}{2}, \TensorPower{\ab}{n})$ vanish unless $(n,k) = (2,0)$ or $(n,k) = (3,1)$. In those two cases we have the following isomorphisms of $(\Q\SymGroup[\bullet], \Q\SymGroup[\bullet])$-bimodules $\Ext^0_{\Foutpol}(\TensorPower{\ab}{2}, \TensorPower{\ab}{2}) \cong \Q\SymGroup[2]$, respectively $\Ext^1_{\Foutpol}(\TensorPower{\ab}{2}, \TensorPower{\ab}{3}) \cong \Specht{(2)} \boxtimes \Specht{(1,1,1)}$.
\end{prop}

\begin{figure}[h]
    \centering
    \includegraphics[width=0.5\linewidth]{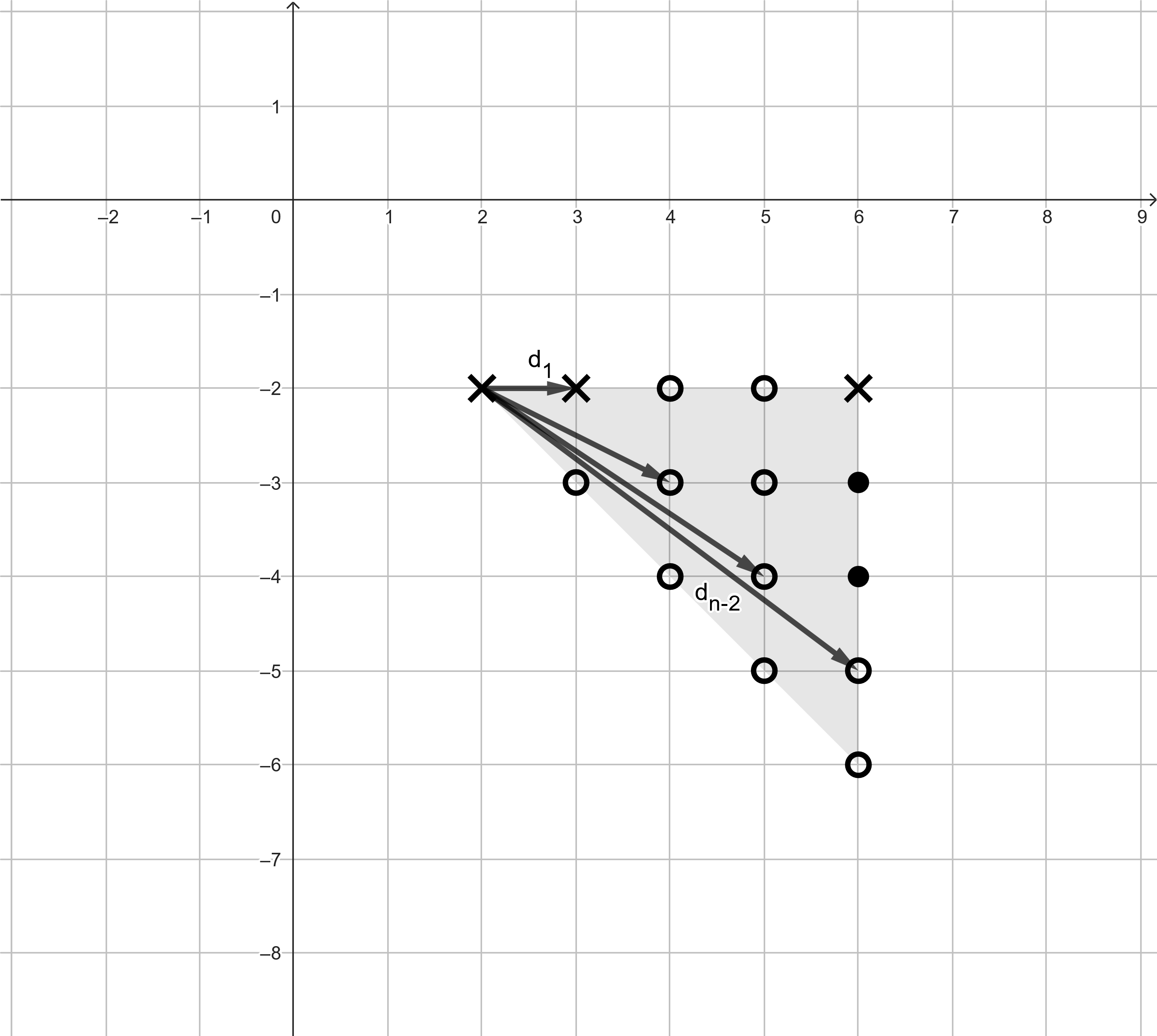}
    \caption{The spectral sequence for computing $\Ext^*_{\Foutpol}(\TensorPower{\ab}{2},\TensorPower{\ab}{n})$, here with $n=6$}
    \label{fig:Spectral-Sequence-a2}
\end{figure}

\begin{proof}
    We already know from Proposition \ref{prop:vanishing-Ext-Out-high-degree} that these groups vanish if $k> n-2$, so in particular in order to have a non-vanishing term we need $n\geq 2$. For $n=2$ we have that $\Ext^k_{\Foutpol}(\TensorPower{\ab}{2}, \TensorPower{\ab}{2})$ vanish outside the degree $k=0$, where it takes the value
    \begin{equation*}
        \Ext^0_{\Foutpol}(\TensorPower{\ab}{2}, \TensorPower{\ab}{2}) = \Specht{(1,1)}\boxtimes\Specht{(1,1)} + \Specht{(2)}\boxtimes \Specht{(2)} \cong \Q\SymGroup[2].
    \end{equation*}

    For $n=3$ we have that $\Ext^k_{\Foutpol}(\TensorPower{\ab}{2}, \TensorPower{\ab}{3})$ vanishes unless $k=1$, where it takes the value
    \begin{equation*}
        \Ext^1_{\Foutpol}(\TensorPower{\ab}{2}, \TensorPower{\ab}{3}) = \Specht{(2)}\boxtimes\Specht{(1,1,1)}.
    \end{equation*}

    For $n\geq 4$ we prove by induction that $\Ext^k_{\Foutpol}(\TensorPower{\ab}{2}, \TensorPower{\ab}{n})$ vanishes identically. In the spectral sequence \eqref{eq:spec-seq-outer} associated to $F = \TensorPower{\ab}{2}$ and $G = \omega\beta\Q\SymGroup[n]$, the induction hypothesis implies that the only terms that can be non-trivial are $E_1^{n,q} = \Ext^{n+q}_{\Foutpol}(\TensorPower{\ab}{2},\TensorPower{\ab}{n})$, as well as $E_1^{2,-2}$ and $E_1^{3,-2}$. It suffices to prove that $d_1\colon E_1^{2,-2}\to E_1^{3,-2}$ is an isomorphism: indeed, in that case on the $E_2$-page we have that $E_2^{n,q} = E_1^{n,q} = \Ext^{n+q}_{\Foutpol}(\TensorPower{\ab}{2},\TensorPower{\ab}{n})$, and all the other terms are already known to be trivial. Consequently, there can be no other non-trivial differential in the spectral sequence, and hence $E_{\infty}^{n,q} = \Ext^{n+q}_{\Foutpol}(\TensorPower{\ab}{2},\TensorPower{\ab}{n})$. By Lemma \ref{lem:spectral-sequence-vanishes} the $E_{\infty}$ page is trivial, and this finishes the proof. We have already proved in Proposition \ref{prop:triangle-E1} that the differential $d_1\colon E_1^{2,-2}\to E_1^{3,-2}$ is injective. See Figure \ref{fig:Spectral-Sequence-a2} for an illustration. 
    
    Thus, all that remains to do is to show that $d_1\colon E_1^{2,-2}\to E_1^{3,-2}$ is also surjective. We use the isomorphism \eqref{Schur-Weyl-Ext} from Remark \ref{rmk:Schur-Weyl-duality} and prove that this is the case in each of the spectral sequences associated to $F = \alpha\Specht{(1,1)}$ or $F = \alpha\Specht{(2)}$ and $G =  \omega\beta\Specht{\lambda}$ for any partition $\lambda\vdash n$.
    
    We consider first the easier case $F = \alpha\Specht{(1,1)}$. Then $E_1^{3,-2} = 0$, and so clearly $d_1$ is surjective in that case.
    
    For the case $\alpha\Specht{(2)}$, we have:
    \begin{equation*}
        E_1^{2,-2} = (\omega\beta\Specht{\lambda})^{(2)}\quad\text{and}\quad E_1^{3,-2} = (\omega\beta\Specht{\lambda})^{(1,1,1)}.
    \end{equation*}
    It follows from \cite[Propositions 5.16, 5.17]{GH22} that
    \begin{equation*}
        (\omega\beta\Specht{\lambda})^{(2)} \cong (\omega\beta\Specht{\lambda})^{(1,1,1)} \cong \Lie((n))\otimes_{\SymGroup}\Specht{\lambda},
    \end{equation*}
    with $\Lie((n))$ denoting the arity $n$ part of the cyclic Lie operad. Since we already know that $d_1\colon E_1^{2,-2}\to E_1^{3,-2}$ is injective, and we just verified that its source and target are isomorphic and finite-dimensional, it must even be an isomorphism.
\end{proof}

\begin{rmk}
    The indecomposable polynomial outer functors of degree $3$ are $\alpha\Specht{(3)}$, $\alpha\Specht{(2,1)}$, $\omega\beta\Specht{(1,1,1)}$, and $\alpha\Specht{(1,1,1)}$, with $\omega\beta\Specht{(1,1,1)}$ being a non-trivial extension of $\alpha\Specht{(1,1,1)}$ by $\alpha\Specht{(2)}$. It is possible to interpret this last proposition as the statement that, for any $n\geq 4$, the functor $\alpha\Specht{(1,1,1)}$ never appears as summand of $\LeftAdjq{3}\omega\beta\Specht{n}$.
\end{rmk}

\subsection{Open questions and future directions}

We finish this paper by discussing some conjectural patterns that appear in our computations, as well as some ideas that have not been explored here but might provide valuable insights. In the first part, we compare the $(\Q\SymGroup[n], \Q\SymGroup[n+k])$-bimodule structure of $\chi(\Ext^*_{\Foutpol}(\TensorPower{\ab}{n},\TensorPower{\ab}{n+k}))$ with that of $\Ext^k_{\CatFuncGr}(\TensorPower{\ab}{n},\TensorPower{\ab}{n+k})$, for varying $n$ and $k$. In particular, for $k = 2$ we obtain a conjectural formula expressing $\chi(\Ext^*_{\Foutpol}(\TensorPower{\ab}{n},\TensorPower{\ab}{n+2})) = \Ext^2_{\Foutpol}(\TensorPower{\ab}{n},\TensorPower{\ab}{n+2})$ in terms of $\Ext^l_{\CatFuncGr}(\TensorPower{\ab}{n},\TensorPower{\ab}{n+l})$ for $l\leq 2$. At the end of this first part we also discuss a possible identification of the differentials in our spectral sequence with cohomology operations involving the terms $\Ext^*_{\Foutpol}(\TensorPower{\ab}{n},\TensorPower{\ab}{m})$ (Yoneda product, as well as Massey products of various orders). The second part discusses a structure of (graded) PROP on $\Ext^*_{\Foutpol}(\TensorPower{\ab}{n},\TensorPower{\ab}{m})$ for varying $n$ and $m$, and we conjecture in particular the existence of certain stabilization properties for this PROP.

We start by presenting patterns observed from our computation of $\Ext^2_{\Foutpol}(\TensorPower{\ab}{n},\TensorPower{\ab}{n+2})$ from Corollary \ref{cor:computing-Ext2}:

\begin{rmk}\label{rmk:numerical-coincidence-Ext2}
    Using the short exact sequence \eqref{eq:formula-Ext2}, we notice that for $2\leq n\leq 8$ there is an identity in the Grothendieck group of $(\Q\SymGroup[n],\Q\SymGroup[n+2])$-bimodules:
    \begin{equation}\label{eq:formula-Ext2-Grothendieck}
    \begin{split}
        [\Ext^2_{\Foutpol}(\TensorPower{\ab}{n},\TensorPower{\ab}{n+2})] = &[\Ext^2_{\CatFuncGr}(\TensorPower{\ab}{n},\TensorPower{\ab}{n+2})] - [\operatorname{Ind}_{\SymGroup[n]^{op}\times\SymGroup[n+1]}^{\SymGroup[n]^{op}\times\SymGroup[n+2]}\Ext^1_{\CatFuncGr}(\TensorPower{\ab}{n},\TensorPower{\ab}{n+1})] +\\
        &[\operatorname{Ind}_{\SymGroup[n]^{op}\times\SymGroup[n]}^{\SymGroup[n]^{op}\times\SymGroup[n+2]}\Ext^0_{\CatFuncGr}(\TensorPower{\ab}{n},\TensorPower{\ab}{n})] - [\Specht{(n)}\boxtimes\Specht{(n+2)}] - [\Specht{(1^{n})}\boxtimes\Specht{(1^{n+2})}];
    \end{split}
    \end{equation}
    in this formula the induction from $\SymGroup[n+i]$ to $\SymGroup[n+2]$ is performed using the trivial representation on $\SymGroup[2-i]$. This formula is almost valid for $n=0$ and $n=1$, except that in those two cases the last term is absent.

    The ranks $\dim \Ext^2_{\Foutpol}(\TensorPower{\ab}{n},\TensorPower{\ab}{n+2})$ are given in Table \ref{tab:rank_Ext2} above. In the range $2\leq n\leq 8$ those ranks satisfy
    \begin{equation}\label{eq:formula-rank-Ext2}
        \dim \Ext^2_{\Foutpol}(\TensorPower{\ab}{n},\TensorPower{\ab}{n+2}) = \operatorname{Rk}(n+2,n)-1,
    \end{equation}
    where $\operatorname{Rk}(b,a)$ is a function defined by Powell and Vespa in \cite{PV23}, and they claim:
    \begin{equation*}
        \dim \Ext^1_{\Foutpol}(\TensorPower{\ab}{n}, \TensorPower{\ab}{n+1}) = \operatorname{Rk}(n+1,n).
    \end{equation*}
\end{rmk}

We state explicitly the following conjecture.

\begin{conj}
    The identities \eqref{eq:formula-Ext2-Grothendieck} and \eqref{eq:formula-rank-Ext2} hold for all $2\leq n$.
\end{conj}

It might be possible to prove this conjecture using \cite[Conjecture 6.7]{GH22} (now a theorem of Powell \cite{PowCoadjoint}). However doing so would be of limited interest, since in our setting these identities are mere numerical coincidences, and other techniques are needed to explain them.

It would be interesting to generalize Equation \eqref{eq:formula-Ext2-Grothendieck} to the case of $\Ext^k_{\Foutpol}(\TensorPower{\ab}{n},\TensorPower{\ab}{n+k})$, for $k\geq 3$. In the scope of this paper, the best we can hope for is to obtain a formula for the Euler characteristic $\chi(\Ext^*_{\Foutpol}(\TensorPower{\ab}{n},\TensorPower{\ab}{n+k}))$, since we have no reason to believe that $\Ext^*_{\Foutpol}(\TensorPower{\ab}{n},\TensorPower{\ab}{n+k})$ has the Koszul property. Moreover, the formula we obtain below is not as clean as in the case $k=2$, as it requires some additional correction terms for small values of $n$:

\begin{rmk}
    For any $k\geq 0$ and any $n$ in the range $n+k\leq 10$, the identity \eqref{eq:formula-Ext2-Grothendieck} can be generalized to the following approximation of $\chi(\Ext^*_{\Foutpol}(\TensorPower{\ab}{n},\TensorPower{\ab}{n+k}))$:
\begin{equation}\label{eq:formula-Extk-Grothendieck}
\begin{split}
    \chi(\Ext^*_{\Foutpol}(\TensorPower{\ab}{n},\TensorPower{\ab}{n+k})) \approx &\sum_{i=0}^k{(-1)^i [\operatorname{Ind}_{\SymGroup[n]^{op}\times\SymGroup[n+k-i]}^{\SymGroup[n]^{op}\times\SymGroup[n+k]}\Ext^{k-i}_{\CatFuncGr}(\TensorPower{\ab}{n},\TensorPower{\ab}{n+k-i})]} \\
    & + (-1)^{k+1}[\Specht{(n)}\boxtimes\Specht{(n+k)}] + (-1)^{k+1}M_{n,k},
\end{split}
\end{equation}
with $M_{n,k}$ being trivial if $k$ is odd or $n < 2$, and otherwise $M_{n,k} = [\Specht{(1^n)}\boxtimes\Specht{(1^{n+k})}]$. More precisely, our computations suggest that the formula above is true when $k\leq 2$ or when $n$ is sufficiently large, while if $k\geq 3$ and $n$ is in a certain range (depending on $k$) one needs to add some sporadic correction terms to the formula\footnote{A list of these correction terms for $n+k\leq 10$ can be found in the following GitHub repository: \url{https://github.com/louishainaut/Ext-Outer-Functors}.}.
\end{rmk}

Another observation coming from Corollary \ref{cor:computing-Ext2} is the following conjectural description of the differentials of the spectral sequence considered for this paper:

\begin{rmk}
    Using Proposition \ref{prop:values-Ext0-1}, one can rewrite the short exact sequence \eqref{eq:formula-Ext2} to express $(\omega\beta\Q\SymGroup[n+2])^{[n]}$ as the following kernel:
    \begin{equation*}
         \ker\Big(\Ext^1_{\Foutpol}(\TensorPower{\ab}{n},\TensorPower{\ab}{n+1})\otimes_{\SymGroup[n+1]}\Ext^1_{\Foutpol}(\TensorPower{\ab}{n+1},\TensorPower{\ab}{n+2}) \to \Ext^2_{\Foutpol}(\TensorPower{\ab}{n},\TensorPower{\ab}{n+2})\Big).
    \end{equation*}
    The reader may have expected the $\Ext$ terms to be written in the opposite order, based on the convention that functions act on the left; however our convention that the place permutation action of the symmetric group acts on the right led us to the stated order. It is natural to wonder whether the map above is actually the Yoneda product. More generally, a mutual induction on $(\omega\beta\Q\SymGroup[n+r])^{[n]}$ and on the differentials $d_r$ suggests the following conjecture:
\end{rmk}

\begin{conj}
    The term $(\omega\beta\Q\SymGroup[n+r])^{[n]}$ is isomorphic to the subspace of 
    \begin{equation}\label{eq:tensor-product-Ext1}
        \Ext^1_{\Foutpol}(\TensorPower{\ab}{n},\TensorPower{\ab}{n+1})\otimes_{\SymGroup[n+1]}\ldots\otimes_{\SymGroup[n+r-1]}\Ext^1_{\Foutpol}(\TensorPower{\ab}{n+r-1},\TensorPower{\ab}{n+r})
    \end{equation}
    consisting of all the elements having the property that all the Yoneda products, as well as all the Massey products up to order $r$, vanish.

    Moreover, in the spectral sequence \eqref{eq:spec-seq-outer} associated with $F = \TensorPower{\ab}{n}$ and $G = \omega\beta\Q\SymGroup[m]$, all the differentials $d_r$ are induced by $(r+1)$-fold Massey products (with the convention that the $2$-fold Massey product is the Yoneda product).
\end{conj}

Writing down explicitly the spectral sequence \eqref{eq:spec-seq-outer} for $F = \TensorPower{\ab}{n}$ and $G = \omega\beta\Q\SymGroup[n+r]$, it is not too hard to prove by an induction on $r$ that $(\omega\beta\Q\SymGroup[n+r])^{[n]}$ is isomorphic to a certain subspace of \eqref{eq:tensor-product-Ext1}, using that $(\omega\beta\Q\SymGroup[n+r])^{[n]} \cong E_1^{n, -n}$ and that $d_1\colon E_1^{n, -n}\to E_1^{n+1,-n}$ is injective. The non-trivial part of this conjecture is to describe this subspace, and doing so requires more control on the differentials than we currently have.

We finish this section with another possible direction for future research. Observe that $\Ext^*_{\Foutpol}(\TensorPower{\ab}{n},\TensorPower{\ab}{m})$ has a structure of graded PROP. Explicitly, there exists a category whose objects are the natural numbers $\N$ and whose morphisms from $m$ to $n$ are $\Ext^*_{\Foutpol}(\TensorPower{\ab}{n},\TensorPower{\ab}{m})$. The vertical composition is the Yoneda product
\begin{equation*}
    \Ext^i_{\Foutpol}(\TensorPower{\ab}{n_1},\TensorPower{\ab}{n_2})\otimes_{\SymGroup[n_2]}\Ext^j_{\Foutpol}(\TensorPower{\ab}{n_2},\TensorPower{\ab}{n_3})\to \Ext^{i+j}_{\Foutpol}(\TensorPower{\ab}{n_1},\TensorPower{\ab}{n_3})
\end{equation*}
and the horizontal product is obtained from \textit{concatenation}
\begin{equation*}
    \Ext^i_{\Foutpol}(\TensorPower{\ab}{l},\TensorPower{\ab}{m})\boxtimes\Ext^j_{\Foutpol}(\TensorPower{\ab}{p},\TensorPower{\ab}{n})\to \Ext^{i+j}_{\Foutpol}(\TensorPower{\ab}{l+p},\TensorPower{\ab}{m+n}).
\end{equation*}

We can construct this concatenation morphism as follows: for any integer $r$ let $P^{\bullet}_{\TensorPower{\ab}{r}}$ be a chosen projective resolution of $\TensorPower{\ab}{r}$. (This is not exactly correct since $\Foutpol$ does not have enough projectives. See the proof of Proposition \ref{prop:description-spectral-sequence}.) Then the totalization $\operatorname{Tot}(P^{\bullet}_{\TensorPower{\ab}{l}}\otimes P^{\bullet}_{\TensorPower{\ab}{p}})$ is a projective resolution of $\TensorPower{\ab}{l+p}$, and the concatenation morphism is induced by the quotient map
\begin{equation*}
    \operatorname{Tot}(P^{\bullet}_{\TensorPower{\ab}{l}}\otimes P^{\bullet}_{\TensorPower{\ab}{p}})^{i+j}\onto P^i_{\TensorPower{\ab}{l}}\otimes P^j_{\TensorPower{\ab}{p}}.
\end{equation*}
We omit the verification that the concatenation is a well-defined operation.

\begin{rmk}
    Vespa showed \cite[Proposition 3.5]{Ves18} that the corresponding graded PROP for $\CatPolyFunc{\mathrm{pol}}$ is freely generated by the operad $\mathcal{Q}$, given by $\mathcal{Q}(n) = \Ext^*_{\CatPolyFunc{\mathrm{pol}}}(\ab,\TensorPower{\ab}{n})$.

    The analogous statement in our context is clearly false, since $\Ext^*_{\Foutpol}(\ab,\TensorPower{\ab}{n})$ is trivial unless $n=1$ (since $\ab$ is projective in $\Foutpol$). Instead, one could investigate whether $\Ext^*_{\Foutpol}(\TensorPower{\ab}{n},\TensorPower{\ab}{n+k})$ exhibit stability phenomena, ie whether, for $n$ large enough, the map
    \begin{equation*}
        \Ext^*_{\Foutpol}(\TensorPower{\ab}{n},\TensorPower{\ab}{n+k}) \to \Ext^*_{\Foutpol}(\TensorPower{\ab}{n+1},\TensorPower{\ab}{n+k+1}),
    \end{equation*}
    obtained by concatenation (say, on the left) with $\Id\in \Ext^0_{\Foutpol}(\ab,\ab)$, is injective and the $(\Q\SymGroup[n+1],\Q\SymGroup[n+k+1])$-bimodule generated by the image is the whole target. Our computations suggest that this might be the case for $\Ext^*_{\Foutpol}(\TensorPower{\ab}{n},\TensorPower{\ab}{n+2})$, with stability region $n\geq 4$.
\end{rmk}

\printbibliography

\end{document}